%% file: expositoryarcs.tex
\documentclass[11]{amsart}

\author{Simeon Ball, Michel Lavrauw}
\thanks{The first author acknowledges the support of the project MTM2017-82166-P of the Spanish {\em Ministerio de Econom\'ia y Competitividad}. The second author acknowledges the support of {\em The Scientific and Technological Research Council of Turkey}, T\"UB\.{I}TAK (project no. 118F159)}

\title{Arcs in finite projective spaces}

 \usepackage{amsmath}
\usepackage{amssymb}
\usepackage{amsthm}
\usepackage{graphicx}
\usepackage{amscd}
\usepackage{color}

 \usepackage{tikz}
\usetikzlibrary{graphs}
\usetikzlibrary{graphs.standard}

\usepackage{pgf,tikz}
\usepackage{mathrsfs}
\usetikzlibrary{arrows}
\usepackage{wrapfig}
 \usepackage{tikz}
\usetikzlibrary{graphs}
\usetikzlibrary{graphs.standard}
\usetikzlibrary{calc,intersections,through,backgrounds}
\usepackage{tkz-euclide}

\usepackage{graphicx}
 \usepackage{url}
\usepackage{eurosym}

\newtheorem{theorem}{Theorem}
\newtheorem{lemma}[theorem]{Lemma}
\newtheorem{remark}[theorem]{Remark}
 \newtheorem{corollary}[theorem]{Corollary}
 
 \newtheorem{conjecture}[theorem]{Conjecture}
\newtheorem{example}[theorem]{Example}

\def\bF{{\mathbb{F}}}
\def\PG{{\mathrm{PG}}}
\def\PGF{{\mathrm{PG}}(k-1,\bF)}
\def\PGL{{\mathrm{PGL}}}

\def\det{{\mathrm{det}}}
\def\cC{{\mathcal{C}}}
\def\cZ{{\mathcal{Z}}}
\def\cN{{\mathcal{N}}}
\def\cA{{\mathcal{A}}}
\def\cG{{\mathcal{G}}}

\begin{document}


\date{23 November 2019}

\maketitle


\begin{abstract}
This is an expository article detailing results concerning large arcs in finite projective spaces, which attempts to cover the most relevant results on arcs, simplifying and unifying proofs of known old and more recent theorems. The article is mostly self-contained and includes a proof of the most general form of Segre's lemma of tangents and a short proof of the MDS conjecture over prime fields based on this lemma.
\end{abstract}

\setcounter{tocdepth}{1}
\tableofcontents


\section{Introduction}

An {\em arc} in a $(k-1)$-dimensional projective space is a set of points with the property that any $k$ of them span the whole space. An arc in a projective plane is called a {\em planar arc}. 
Arcs in finite projective spaces have a long history dating back at least as far as the 1950's and 1960's when Beniamino Segre published his fundamental work on arcs, including the articles \cite{Segre1955a}, \cite{Segre1955b}, \cite{Segre1957}, \cite{Segre1962} and \cite{Segre1967}. Segre's work included a proof of a conjecture of J\"arnefelt and Kustaanheimo \cite{JK1952} that a planar arc of size $q+1$, $q$ odd, is a conic.  See \cite[footnote (5) p. 359]{Segre1955b} for an interesting historical account of this result. 

Segre highlighted three problems concerning arcs. One of these problems ({\em Problema $II_{r,q}$} in \cite{Segre1955b}) is whether there exist parameters $(k,q)$ for which every arc of size $q+1$ is a normal rational curve. This was answered in the affirmative by Segre's theorem for $k=3$ and $q$ odd, but a full classification of such parameters has not been achieved yet. A second problem  ({\em Problema $III_{r,q}$} in \cite{Segre1955b}) concerned the extendibility of arcs in projective spaces with such parameters, and asks for the possible size of an arc which is not extendable to a normal rational curve. Both of these problems have been the subject of much research in the last sixty years, and some of the main results are collected in this article. However, the most fundamental problem  ({\em Problema $I_{r,q}$} in \cite{Segre1955b}) is that of determining the size and the structure of the largest arcs in a given projective space over a finite field. We still do not have a complete answer to this question, although progress is continuing to be made, and this article aims to provide unified proofs of the main results (including Segre's theorem) and the mathematics which was developed in the process. A solution to this problem would settle the famous {\em MDS conjecture} for maximum distance separable codes, see Section \ref{sec:MDS_conjecture}.

A simple combinatorial approach to these problems related to arcs reveals little. It is possible to improve on the trivial upper bound of $q+k-1$ for the size slightly but any reasonable results on arcs  in $k$-dimensional projective spaces over a finite field with $q$ elements, require an algebraic and geometric approach. Indeed, all the strongest results on arcs are based on Segre's initial ideas to find connections between arcs and algebraic curves/varieties.

Segre's approach relies on his lemma of tangents which will form the basis of this exposition. Note that in this context a {\em tangent} to a planar arc $\cA$ is a combinatorial tangent, i.e. a line meeting the arc $\cA$ in exactly one point, and this should not be confused with the notion of a tangent line of an algebraic curve.  For a planar arc $\cA$ of size $q+1$, $q$ odd, the lemma of tangents says that the triangle, formed by any three distinct points of $\cA$, and the triangle, formed by the tangent lines at these points, are in perspective (\cite[Theorem II]{Segre1955a}).

The most significant methodological contribution in Segre's work was to associate an algebraic curve to a planar arc. In particular, he proved that the set of tangents to a planar arc belong to an algebraic envelope of class $t$ when $q$ is even and of class $2t$ when $q$ is odd (\cite[Theorem I, II]{Segre1959}), where the constant $t$ is the number of tangents through a point of the arc. The crucial part in the proof of the existence of this envelope is \cite[Theorem I]{Segre1959}, which can be seen as a generalisation of the theorem of Menelaos (which gives a simple condition for collinearity of three points on the sides of a triangle). Segre's envelope associated to a planar arc allowed the use of existing bounds on the number of points on algebraic curves over finite fields to obtain bounds for the size of planar arcs by careful analysis of its irreducible components, see for example \cite{Voloch1991}, \cite{HiKo1996}, \cite{HiKo1998} and \cite{GiPaToUg2002}.

More recently, elaborating on Segre's original ideas, reformulations and generalisations of his lemma of tangents have simplified the lemma which has extended its applicability. We refer to Section \ref{sec:lemma_of_tangents} for the {\em Coordinate-free lemma of tangents} from \cite{Ball2012} and the most recent {\em Scaled coordinate-free lemma of tangents} from \cite{BL2019}. These lemmas have led to significant advancement in our knowledge of large arcs, not least that we now know how large an arc can be when $q$ is prime \cite{Ball2012}. A short proof of this result can be found in Section \ref{sec:MDS_p}. The main motivation of the article was to set down results that can be derived using the afore-mentioned versions of Segre's lemma of tangents, which allow us to simplify proofs of many recent and old results.

Arcs appear in many branches of mathematics. In problems of finite geometry itself they are prevalent. To represent a matroid of rank $r$ over a finite field one must map subsets of the matroid whose $r$-subsets are independent sets onto an arc, so the relevance of arcs to the representablility of matroids is evident, see \cite{GRZ2018}. Non-existence results on arcs translate to non-representability results for matroids. 
An application of arcs to the theory error correcting codes, is given by the well-established fact that a $k$-dimensional linear maximum distance separable code is entirely equivalent to an arc in $\mathrm{PG}(k-1,q)$. This is detailed in Section \ref{sec:MDS_conjecture}. The famous MDS conjecture, which is a focus of this article, comes from this view of arcs. MDS codes are the codes of choice for codes over large alphabets and there has been a recent resurgence of interest in MDS codes and locally recoverable codes (see \cite{BT2014} and \cite{GXY2019}) due to their applications to distributed storage systems.
There are further applications of arcs to quantum physics, see \cite{Bernal2017} and \cite{GR2015}, to cryptography \cite {FPXY2014}, to group theory \cite{CL2012} and to other combinatorial objects \cite{Potapov2018} and \cite{WLL2018}. 

The articles we have referenced in these paragraphs are by no means exhaustive and are just given as an example of the far reaching relevance of arcs of finite projective spaces. 

\section{Basic objects and definitions}

A point of $\mathrm{PG}(k-1,{\mathbb F})$ is a one-dimensional subspace of $\bF^k$, and its {\em coordinate vector} with respect to a basis, is determined up to a nonzero scalar and
is denoted by 
$$(x_1,\ldots,x_k).$$

The {\em weight} of a vector (or a point) with respect to some basis is the number of non-zero coordinates it has with respect to that basis. 

An arc of $\PGF$ of size $k+1$ is called a {\em frame} of $\PGF$. An arc is called {\em complete} if it is not contained in a larger arc. An arc of size $\leq k$ in $\PGF$ is always incomplete. Often we will implicitly assume that an arc has size at least $k+1$.

Let ${\mathrm{PGL}}(k,\bF)$ denote the projective linear group. Two sets of points $A$ and $B$ in $\PGF$ are called {\em projectively equivalent} if there exists a {\em projectivity} $\varphi \in {\mathrm{PGL}}(k,\bF)$ such that $B=A^\varphi$.

A symmetric bilinear form $b(X,Y)$ on $\bF^k$ (where $X$ and $Y$, stand for $X_1,\ldots,X_k$ and $Y_1,\ldots,Y_k$) is {\em degenerate} if $b(X,y)=0$ for some point $y$. A quadratic form $f(X)$ is {\em degenerate} if $f(y)=0$ and $b(X,y)=0$ for some point $y$, where 
$$b(X,Y)=f(X+Y)-f(X)-f(Y)$$ 
is the bilinear form associated to $f(X)$. 
The zero-set in $\PG({2},\bF)$ of a (non-degenerate) quadratic form on $\bF^3$ is called a {\em (non-degenerate) conic}. We note that there is a unique conic through an arc of size 5 in $\mathrm{PG}({2},\bF)$.


\section{Normal rational curves}

A {\em normal rational curve} is a set of points in $\PG(k-1,\bF)$, 
projectively equivalent to
\begin{eqnarray}\label{eqn:normal rational curve}
\cN_{k-1}=\{ (1,t,\ldots,t^{k-1}) \ | \ t \in {\mathbb \bF}\} \cup \{ (0,\ldots,0,1) \}.
\end{eqnarray}

\begin{lemma}
A normal rational curve in $\PG(k-1,\bF)$, with $|\bF|\geq k-1$, is an arc of $\PG(k-1,\bF)$.
\end{lemma}
\begin{proof}
The condition that $k$ points of the normal rational curve as defined in (\ref{eqn:normal rational curve}) are linearly independent is equivalent to a Vandermonde determinant being different from zero.
\end{proof}

\begin{remark}
If the condition $|\bF|\geq  k-1$ is not satisfied, say $\bF=\bF_q$ with $q< k-1$, then the normal rational curve $\cN_{k-1}$ has $q+1<k$ points. Therefore $\cN_{k-1}$ spans at most a hyperplane of $\PG(k-1,\bF)$. In this case $t^{k-2}=t^i$ for some $1\leq i<k-2$ for all $t\in \bF$ and $\cN_{k-1}$ defines a normal rational curve in the subspace spanned by $\cN_{k-1}$.
\end{remark}

For $k=3$, a normal rational curve is the zero-set of a quadratic form, i.e. it is a conic. In the example (\ref{eqn:normal rational curve}) above, the quadratic form is $X_1X_3-X_2^2$.

Since an arc of size $k+1$ is a {\em frame} of $\PGF$ and $\PGL(k,\bF)$ acts regularly on the set of frames of $\PGF$, any two arcs of size $k+1$ are projectively equivalent.

The normal rational curve as defined in (\ref{eqn:normal rational curve}) is the image in $\PG(k-1,\bF)$ of the set of points of $\PG(1,\bF)$ under the polynomial map 
\begin{eqnarray}\label{eqn:map1}
\nu_{k-1}~:~(x_1,x_2)\mapsto (x_1^{k-1},x_1^{k-2}x_2,\ldots,x_2^{k-1}).
\end{eqnarray}
The map $\nu_{k-1}$ can be used to define an action of the group $\PGL(2,\bF)$ on $\PGF$ as follows.

\begin{lemma}\label{lem:tilde}
Given $\alpha\in \PGL(2,\bF)$, there is a projectivity $\tilde \alpha \in \PGL(k,\bF)$ satisfying
$\nu_{k-1}(x)^{\tilde{\alpha}}=\nu_{k-1}(x^\alpha)$ for all points $x$ of $\PG(1,\bF)$.
Moreover, if $|\bF|\geq k$, then $\tilde \alpha$ is unique.
\end{lemma}
\begin{proof}
It is straightforward to verify that for each $\alpha\in\PGL_2(\bF)$, the map
$\nu_{k-1}(x)\mapsto \nu_{k-1}(x^\alpha)$ defines a projectivity $\tilde \alpha$ of $\PG(k-1,\bF)$. The uniqueness follows from the fact that if $|\bF|\geq k$ then the image of $\nu_{k-1}$ contains a frame and a projectivity is uniquely determined by its action on a frame.
\end{proof}
The action of the group $\PGL(2,\bF)$ on $\PGF$ is then defined by lifting the action on $\PG(1,\bF)$:
\begin{eqnarray}\label{eqn:lift}
{\tilde{\varphi}}:\PGL(2,\bF)\rightarrow  \PGL(k,\bF): \alpha\mapsto {\tilde{\alpha}},
\end{eqnarray}
where $\tilde \alpha$ is the projectivity induced by $\alpha$ as in the proof of Lemma \ref{lem:tilde}.
\begin{lemma}\label{lem:normal rational curve_transitivity}
The image of $\tilde \varphi$ is a subgroup $H\cong \PGL(2,\bF)$ of $\PGL({k-1},\bF)$ and the action defined by the restriction of $H$ to the points of the normal rational curve $\cN_{k-1}$ from (\ref{eqn:normal rational curve}) is faithful. In particular, $H$ acts regularly on the set of triples of distinct points of $\cN_{k-1}$.
\end{lemma}
\begin{proof}
The map $\tilde \varphi$ defined in (\ref{eqn:lift}) is a group homomorphism with trivial kernel.
\end{proof}
\begin{lemma}\label{lem:unique_normal rational curve}
If $|\bF|\geq k+1$ then there is a unique normal rational curve through an arc in $\PG(k-1,\bF)$ of size $k+2$.
\end{lemma}
\begin{proof}
Consider an arc in $\PG(k-1,\bF)$ consisting of the $k+2$ points $p_0,\ldots,p_{k+1}$. Suppose $p_0=(u_1,\ldots,u_k)$ and $p_{k+1}=(v_1,\ldots,v_k)$ with respect to the basis $p_1,\ldots,p_{k}$. Observe that this implies that the vectors $(u_i,v_i)$, $i=1,\ldots, k$, define $k$ distinct points in $\PG(1,\bF)$.
Now consider the polynomial map 
$$\varphi_f:(x_1,x_2)\mapsto (f_1(x_1,x_2),\ldots, f_k(x_1,x_2))$$ 
from $\PG(1,\bF)$ to $\PG(k-1,\bF)$, where 
$$
f_i(X_1,X_2)= \prod_{j=1,j\neq i}^k (u_j^{-1}X_1-v_j^{-1}X_2),~~~i=1,\ldots,k
$$
are forms of degree $k-1$. Since $\sum a_if_i=0$ implies $a_i=0$ after substituting $(u_i,v_i)$ for $(X_1,X_2)$, the forms $f_1,\ldots,f_k$ form a basis for the space of forms of degree $k-1$ in $\bF[X_1, X_2]$. This implies that up to a change of basis the image of $\varphi_f$ coincides with that of $\nu_{k-1}$ as in (\ref{eqn:map1}), and is therefore a normal rational curve  $\cN_f$ in $\PG(k-1,\bF)$. The image under $\varphi_f$ of the point $(u_l,v_l)$, for $l\in\{1,\ldots,k\}$, is the point $p_l$, since $f_i(u_l,v_l)=0$ for $i\neq l$ and $f_i(u_i,v_i)=\prod_{j\neq i} (u_j^{-1}u_i-v_j^{-1}v_i)$, which is nonzero.
The image under $\varphi_f$ of the points $(1,0)$ and $(0,1)$ are the points $p_0$ and $p_{k+1}$. This shows that the points $p_0,\ldots,p_{k+1}$ are contained in a normal rational curve in $\PG(k-1,\bF)$.

The uniqueness follows by considering any normal rational curve $\cN_g$ given by the image of 
$$\varphi_g:(x_1,x_2)\mapsto (g_1(x_1,x_2),\ldots, g_k(x_1,x_2))$$ 
where the $g_i$'s form a basis for the space of forms of degree $k-1$ in $\bF[X_1,X_2]$, and assuming that $\cN_g$ contains the $k+2$ points $p_0,\ldots,p_{k+1}$. W.l.o.g. we may assume that the points $(1,0)$ and $(0,1)$ are the pre-images of the points $p_0$ and $p_{k+1}$ (if necessary, apply a coordinate transformation to $\PG(1,\bF)$ and relabel the $g_i$'s).
For each $i\in \{1,\ldots,k\}$, the assumption $p_i\in \cN_g$ implies the existence of a
point $(a_i,b_i)$ in $\PG(1,\bF)$ and a factor $a_i^{-1}X_1-b_i^{-1}X_2$ of each $g_j$, $j\neq i$. This determines each of the $g_i$'s up to a nonzero scalar factor. The condition that $\varphi_g(1,0)=p_0$ and $\varphi_g(0,1)=p_{k+1}$ implies that 
$(u_1,\ldots,u_k)=\lambda(a_1,\ldots,a_k)$ and $(v_1,\ldots,v_k)=\mu(b_1,\ldots,b_k)$ and therefore
$$
g_i(X_1,X_2)=\prod_{j=1,j\neq i}^k (u_j^{-1}\lambda X_1-v_j^{-1}\mu X_2),~~~i=1,\ldots,k.
$$
If follows that $\cN_f=\cN_g$, since $\cN_g=\cN_g^{\tilde\alpha}$ where $\alpha$ is the projectivity of $\PG(1,\bF)$ mapping $(x_1,x_2)$ to $(\lambda x_1,\mu x_2)$ (see Lemma \ref{lem:tilde} and Lemma \ref{lem:normal rational curve_transitivity}).
\end{proof}

\begin{lemma}\label{lem:projection_normal rational curve} The projection $x(\cN)$ of a normal rational curve $\cN$ in $\PG(k-1,\bF)$ from one of its points $x$ onto a hyperplane $\Pi$, $x\notin \Pi$, is contained in a normal rational curve of $\Pi$, which is uniquely determined by $x(\cN)$ if $|\bF|\geq k+1$.
\end{lemma}
\begin{proof}
Let $\cA$ be the set of points on the normal rational curve as defined in (\ref{eqn:normal rational curve}). The projection from the point $(1,0,\ldots,0)$ onto the hyperplane $\Pi$ with equation $X_1=0$ is the set of points $\{ (0,1,t,\ldots,t^{k-2}) \ | \ t \in \bF\setminus\{0\}\} \cup \{ (0,\ldots,0,1) \}$, to which we can add the point $(0,1,0,\ldots,0)$ to obtain a normal rational curve in $\Pi$.
Applying Lemma \ref{lem:unique_normal rational curve} and the transitivity properties   from Lemma \ref{lem:normal rational curve_transitivity} concludes the proof.
\end{proof}

A normal rational curve is a nonsingular algebraic curve in $\PG(k-1,\bF)$ and the {\it tangent line} at the point $p=\nu_{k-1}(x_0,x_1)$ of $\cN_{k-1}$ is the line $\ell_p=\langle p, p_0,p_1\rangle$ where 
$p_i$ is the point whose coordinates are the evaluations at $(x_0,x_1)$ of the partial derivatives of the coordinates of $\nu_{k-1}(X_0,X_1)$ with respect to $X_i$ ($i=0,1$).
\begin{lemma}\label{lem:tangents_normal rational curve} 
The tangents to a normal rational curve $\cN_k$, with $k\geq 4$, are pairwise disjoint and only meet the secants of $\cN_k$ which pass through the points of tangency. 
\end{lemma}
\begin{proof}
Observing that the tangent lines at the points $\nu_{k-1}(1,0)$ and $\nu_{k-1}(0,1)$ are disjoint, the first part of the lemma follows from the transitivity properties from Lemma \ref{lem:normal rational curve_transitivity}. For the second part, consider the tangent line $\ell_p$ at the point $\nu_{k-1}(1,0)$ and the secant line $\ell$ through the points $\nu_{k-1}(0,1)$ and $\nu_{k-1}(1,1)$. Observing that the matrix
$$
\begin{bmatrix}
1&0&0&\ldots&0&0\\
0&1&0&\ldots&0&0\\
0&0&0&\ldots&0&1\\
1&1&1&\ldots&1&1\\
\end{bmatrix}
$$
has rank 4, implies that the two lines $\ell_p$ and $\ell$ are disjoint.
\end{proof}

\begin{lemma}\label{lem:extra_tangents_proj_normal rational curve}
If $x$ is a point of a normal rational curve $\cN=\cN_{k-1}$ in $\PG(k-1,\bF)$ with $|\bF|\geq k+1$, $\Pi$ a hyperplane not passing through $x$, and $\cN_x$ is the unique normal rational curve passing through the projection $x(\cN)$ of $\cN$ from $x$ onto $\Pi$, then the tangent line of $\cN$ at $x$ is the line through $x$ and the unique point of $\cN_x\setminus x(\cN)$.
\end{lemma}
\begin{proof}
By the transitivity properties from Lemma  \ref{lem:normal rational curve_transitivity} we may assume $x=\nu_{k-1}(1,0)$. The result then immediately follows from the details given in the proofs of Lemma \ref{lem:projection_normal rational curve} and Lemma \ref{lem:tangents_normal rational curve}.
\end{proof}

\section{Examples of large arcs in $\PG(k-1,q)$}

From now on we consider the case $\bF=\bF_q$ for some prime power $q$. 
As we have seen in the previous section, when $q\geq k-1$, a normal rational curve is an example of an arc of size $q+1$ in $\PG(k-1,q)$.
If $k=3$ then a normal rational curve is a non-degenerate conic $\cC$, and if $q$ is even, then the tangents to $\cC$ are concurrent. The point $N$ common to all the tangents is called the {\em nucleus} of $\cC$. If follows that $\cC\cup \{N\}$ is also an arc. This gives an example of what is called a {\em hyperoval}: an arc of size $q+2$ in $\PG(2,q)$. So a hyperoval can be obtained from a non-degenerate conic $\cC$ in $\PG(2,q)$, $q$ even, by adding the nucleus of $\cC$. Another example of a hyperoval is the following.

\begin{example}[Segre \cite{Segre1957}] \label{translation}
\rm{Let $\sigma$ be the automorphism of ${\mathbb F}_q$, $q=2^h$, which takes $x$ to $x^{2^e}$. The set
$$
\cA=\{ (1,t,t^{\sigma}) \ | \ t \in {\mathbb F}_q \} \cup \{ (0,0,1) , (0,1,0) \}
$$}
is an arc of $q+2$ points in $\PG(2,q)$, whenever $(e,h)=1$. It is an example of a {\em translation hyperoval}.
\end{example}

\begin{proof}
We prove that $\cA$ defined in Example~\ref{translation} is an arc. Clearly, the line through $(0,0,1)$ and $(0,1,0)$ ($X_1=0$), as well as the lines through one of 
$\{ (0,0,1) , (0,1,0) \}$ and a point $(1,t,t^{\sigma})$ ($tX_1=X_2$ and $t^\sigma X_1=X_3$) contains no other points of $\cA$.
If three distinct points (parameterised by $t_1,t_2,t_3$) of $\cA$ were collinear then, by calculating the corresponding determinant, one obtains 
$$\left ( \frac{t_3-t_1}{t_2-t_1}\right )^{\sigma-1}=1.$$ 
Since $(e,h)=1$ this implies $t_2=t_3$, a contradiction.
\end{proof}

We refer to \cite{Pinneri1996}, \cite{PP1999} and the more recent \cite[Section 1]{Vandendriessche2019} and to Section~\ref{hypsection} in this article, for more on hyperovals. 

Next we give two further examples of large arcs (of size $q+1$) which are not normal rational curves. The first example is an example in $\PG(3,q)$, $q$ even, the second in $\PG(4,q)$, $q=9$. As we will see, such examples are extremely rare.
\begin{example}[Segre \cite{Segre1967}]\label{ex:3space}
\rm{Let $\sigma$ be the automorphism of ${\mathbb F}_q$, $q=2^h$, which takes $x$ to $x^{2^e}$. The set
$$
\cA=\{ (1,t,t^{\sigma},t^{\sigma+1}) \ | \ t \in {\mathbb F}_q \} \cup \{ (0,0,0,1)\}
$$
is an arc of $q+1$ points in $\mathrm{PG}(3,q)$, whenever $(e,h)=1$. Note that $\cA$ is a normal rational curve when $e=1$.}
\end{example}
\begin{proof}
Consider the  map
\begin{eqnarray}\label{eqn:map1}
\varphi_1~:~(x_1,x_2)\mapsto (x_1^{\sigma+1},x_1^{\sigma}x_2,x_1x_2^\sigma,x_2^{\sigma+1})
\end{eqnarray}
from $\PG(1,q)$ to $\PG({3},q)$. Then $\cA$ is the image of $\varphi_1$, and lifting the action of 
$\PGL(2,q)$ on $\PG(1,q)$ to $\PG(k-1,q)$, (in the same way as in (\ref{eqn:lift}), but now with $\varphi_1$ instead of $\varphi$) 
we obtain the action of a subgroup $H\leq \PGL_4(\bF_q)$, $H\cong \PGL_2(q)$, on the set of points of $\cA$, which is regular on the set of triples of distinct points of $\cA$. 
Now consider any four distinct points of $\cA$. Using the action of $H$, we may assume that the first three points are the images under $\varphi_1$ of the points $(1,0)$, $(0,1)$ and $(1,1)$. Let $(1,t,t^{\sigma},t^{\sigma+1})$, $t\in \bF_q\setminus\{0,1\}$ be the fourth point. Then these four points are collinear if and only if $t^\sigma=t$, a contradiction since $t\notin \{0,1\}$ and $t^\sigma=t^{2^e}$ with $(e,h)=1$.
\end{proof}


\begin{example} [Glynn \cite{Glynn1986}] \label{glynn}
\rm{Let $\eta$ be an element of ${\mathbb F}_9$, $\eta^4=-1$. The set
$$
\cA=\{ (1,t,t^{2}+\eta t^6,t^{3},t^4) \ | \ t \in {\mathbb F}_9\} \cup \{ (0,0,0,0,1)\}.
$$
is an arc of size $q+1$ in $\mathrm{PG}(4,9)$.}
\end{example}
\begin{proof}
Consider five distinct points of $\cA$ (for the moment excluding the point $(0,0,0,0,1)$) parameterised by $t_1,\ldots,t_5$. Let $A$ denote the $5\times 5$-matrix, with the corresponding vectors as rows.
The determinant of  $A$ equals $|A_1|+\eta |A_2|$, where $A_1$ and $A_2$ are equal to $A$, except in their third column, where $A_1$ has entries $t_i^2$, and $A_2$ has $\eta t_i^6$, instead of the entries $t_i^2+\eta t_i^6$ in $A$. If $|A|=0$ then $|A_1|=-\eta |A_2|$. Applying the automorphism $\sigma~:~a\mapsto a^3$ to the entries of $A_1$ gives the matrix obtained from $A_2$ after applying the permutation $(24)$ to its columns. This implies that $| A_1 |^3=-| A_2|$ and so $| A_1|^4 =| A_2|^4$, since $| A_1 |^{12}=| A_1 |^4$. It follows from
$|A_1|=-\eta | A_2|$ that $|A_1|^4=\eta^4| A_1|^4$. Since $A_1$ is a Vandermonde matrix, $|A_1|\neq 0$, which contradicts $\eta^4=-1$.\\
If one of the five points is the point $(0,0,0,0,1)$, then the same arguments work, starting from the $4\times 4$ matrix $A$ whose rows correspond to the vectors of $\cA$ parametrised by 4 distinct values of $t$, but with the last coordinate removed.
\end{proof}

\section{The trivial upper bound and the MDS conjecture}\label{sec:MDS_conjecture}

In this section we give the trivial upper bounds on the size of an arc, explain the links between arcs and MDS codes and conclude with the statement of the main conjecture for MDS codes.

\begin{lemma} \label{triv}
If $S$ is a $(k-2)$-subset of an arc $\cA$ in $\PG(k-1,q)$ then there are exactly $t=q+k-|\cA|-1$ 
hyperplanes which meet $\cA$ in precisely $S$. 
\end{lemma}

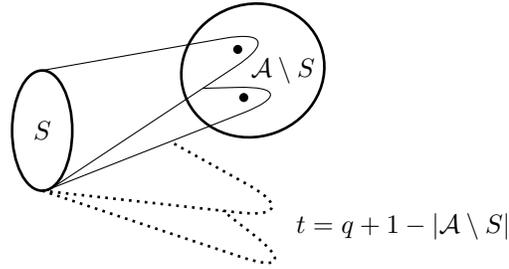
\input{def_t.txt}

\begin{proof}
There are $q+1$ hyperplanes containing the $(k-3)$-dimensional subspace spanned by $S$, each containing at most one point of $\cA \setminus S$. Therefore there are $q+1-|\cA\setminus S|$ hyperplanes which meet $\cA$ in precisely the points of $S$.
\end{proof}

\begin{theorem}\label{cor:trivial_upper_bound}
An arc of $\PG(k-1,q)$ has at most $q+k-1$ points.
\end{theorem}

\begin{proof}
This follows from Lemma~\ref{triv}, since $t \geq 0$.
\end{proof}

\begin{theorem}
Let $\cA$ be an arc of $\PG(k-1,q)$. If $k \geq q$ then $|\cA| \leq k+1$.
\end{theorem}

\begin{proof}
After choosing a suitable basis we may assume that $\cA$ contains the standard frame
$$
\{ e_1,\ldots,e_k,e_1+\cdots+e_k \},
$$
where $e_i$ is the $i$-th vector in the standard basis of $\bF_q^k$.
Suppose $u=(u_1,\ldots,u_k) \in \cA \setminus  \{ e_1,\ldots,e_k,e_1+\cdots+e_k \}$. 
If $u_i=0$ for some $i$ then the hyperplane with equation $X_i=0$ contains $k$ points of $\cA$, contradicting the arc property.
If $u_i \neq 0$ for all $i$ then, since $k\geq q$, by the pigeon-hole principle there exists and $i$ and $j$ such that $u_i=u_j$. But then the hyperplane with equation $X_i=X_j$, contains $k$ points of $\cA$, contradicting the arc property.
\end{proof}

Let $G_\cA$ denote the $k \times |\cA|$ matrix whose columns are the coordinate vectors of the points belonging to a set $\cA$ in $\PG(k-1,q)$.

\begin{lemma} \label{kminus1}
The set $\cA$ is an arc in  $\PG(k-1,q)$ if and only if 
for all non-zero $u \in {\mathbb F}_q^k$ the vector $uG_\cA$ has at most $k-1$ zeros.
\end{lemma}

\begin{proof}
The $i$-th coordinate of $uG_\cA$ is zero if and only if the point of $\cA$ corresponding to the $i$-th column of $G_\cA$ lies in the hyperplane defined by $u$. 
\end{proof}

\begin{lemma}
A non-zero vector in $\{ uG_\cA \ | \ u \in  {\mathbb F}_q^k\}$ has weight at least $|\cA|-k+1$.
\end{lemma} 

\begin{proof}
This follows immediately from Lemma~\ref{kminus1}.
\end{proof}

A $k$-dimensional {\em linear code} of length $n$ and minimum distance $d$ is a $k$-dimensional subspace of ${\mathbb F}_q^n$ in which every non-zero vector has weight at least $d$. Such a code is called an {\em $\bF_q$-linear $[n,k,d]$-code}, where $n$, $k$, and $d$ are ofter referred to as the {\em parameters} of the code.

Let $C$ be a linear code  of length $n$ and minimum distance $d$. If $\bar u$ denotes the vector obtained from the codeword $u\in C$ by deleting the first $d-1$ coordinates from $u$, then clearly for any two codewords $u\neq v$, we must have $\bar u\neq \bar v$.
This implies that the dimension $k$ of $C$ can be at most $n-(d-1)$. Equivalently $d\leq n-k+1$. This bound is called the {\em Singleton bound}.
A $k$-dimensional {\em linear maximum distance separable} (MDS) code $C$ of length $n$ is a $k$-dimensional subspace of ${\mathbb F}_q^{n}$ in which every non-zero vector has weight at least $n-k+1$. We have already established the following theorem.

\begin{theorem} \label{arcisMDS}
The linear code $C$ generated by the matrix $G$, whose columns are the coordinate vectors of the points of an arc is a linear MDS code, and vice versa, the set of columns of a generator matrix of an MDS code considered as a set of points of the projective space, is an arc.
\end{theorem}

The {\em dual} of a linear code $C$ of length $n$ over $\bF_q$ is,
$$
C^{\perp}=\{ v \in {\mathbb F}_q^n \ | \ u \cdot v=0 \ \mathrm{for\ all} \ u \in C \},
$$
where $u \cdot v$ denotes the standard inner product $u \cdot v=u_1v_1+\cdots+u_nv_n$. If $C$ has dimension $k$, then $C^\perp$ has dimension $n-k$.

\begin{lemma} \label{dualMDS}
The linear code $C$ is MDS if and only if $C^{\perp}$ is MDS.
\end{lemma}

\begin{proof}
Suppose $C$ is a $k$-dimensional MDS code of length $n$ and that $C^{\perp}$ is not MDS. Then $C^{\perp}$ contains a non-zero vector $v$ of weight less than $n-(n-k)+1=k+1$. Let $G$ be the generator matrix of $C$. Since $v\in C^\perp$ if and only if $G v=0$, the (at most $k$) non-zero coordinates of $v$ give a non-trivial linear dependence relation between at most $k$ columns of $G_\cA$. This contradicts the fact that the columns of $G$ form an arc in $\PG(k-1,q)$.
\end{proof}

\begin{corollary} \label{dualarc}
There is an arc of size $n$ in $\PG(k-1,q)$ if and only if there is an arc of size $n$ in $\PG(n-k-1,q)$.
\end{corollary}
\begin{proof}
This follows from Theorem~\ref{arcisMDS} and Lemma~\ref{dualMDS}.
\end{proof}

We have seen examples (normal rational curves and some others) of arcs of size $q+1$ in $\PG(k-1,q)$ (for $q\geq k-1$) and of size $q+2$ in $\PG(2,q)$.

If $\cA$ is the set of points on the normal rational curve as defined in (\ref{eqn:normal rational curve}), and $G_\cA$ is the generator matrix of the associated MDS code, as described above, then for each $u=(u_0,\ldots,u_{k-1})\in \bF_q^k$, we obtain a codeword $c=uG_\cA$, whose last coordinate is $u_{k-1}$ and whose other coordinates are of the form
$$u_0+u_1t+u_2t^2+\ldots+u_{k-1}t^{k-1},~ t\in \bF_q.
$$ 
Therefore, the code $C=\{ uG_\cA \ | \ u \in  {\mathbb F}_q^k\}$ can be rewritten as
\begin{eqnarray}\label{eqn:RS_code}
C=\{ (f(t_1),f(t_2),\ldots, f(t_q),f_{k-1}) ~ |~ f (X) =\sum_{i=0}^{k-1}f_iX^i \in  {\mathbb F}_q[X]\}
\end{eqnarray} 
where we have labeled the $q$ elements of $\bF_q$ as $\{t_1,t_2,\ldots,t_q\}$.
This MDS code obtained from a normal rational curve is called the {\em Reed-Solomon} code (or {\em extended} Reed-Solomon code) and has parameters $[q+1,k,q-k+2]$.

\begin{theorem} \label{thm:dual_of_RS}
The dual of the Reed-Solomon code is a Reed-Solomon code.
\end{theorem}
\begin{proof}
Let $C_k$ denote the Reed-Solomon code obtained from the normal rational curve in $\PG(k-1,q)$ as in (\ref{eqn:RS_code}).
Let $c_f\in C_k$ and $c_g\in C_{q+1-k}$ with $f,g \in \bF_q[X]$. The inner product of $c_f$ and $c_g$ is
$$
\sum_{t\in \bF_q}f(t)g(t)+f_{k-1}g_{q-k}.
$$
Since the sum of all $n$-th powers $\{t^n~:~t\in \bF_q\}$ is zero unless $n$ is a multiple of $q-1$ in which case it is $-1$, the above sum is zero.
\end{proof}

By Corollary \ref{dualarc} the normal rational curves also give arcs of size $q+1$ in $\PG(q-k,q)$ and the hyperovals give arcs of size $q+2$ in $\PG(q-3,q)$. If we exclude the trivial parameters, so assume $3\leq k\leq q-2$, and we also exclude the case of hyperovals (including the arcs obtained from the dual codes of a hyperoval code), we obtain the interval $4\leq k\leq q-3$ for the dimension $k$ of an $\bF_q$-linear MDS code. For these dimensions, no $\bF_q$-linear MDS codes are know which have length larger than $q+1$. Equivalently, if $4\leq k\leq q-3$, then no arcs in $\PG(k-1,q)$ are known of size larger than $q+1$. This led to the well-known {\em MDS-conjecture}, see for example \cite[Chapter 11]{MS1977}. 

The MDS conjecture can be extended to non-linear codes but that was not done until recently, see \cite{Huntemann2012}. For the purposes of this article, the MDS conjecture will refer to the MDS conjecture for linear codes.

\begin{conjecture}\label{conject:MDS} (The MDS conjecture) If $4 \leq k \leq q-3$ then an arc in $\PG(k-1,q)$ has size at most $q+1$.
\end{conjecture}

In fact, although we have already seen some examples (precisely Examples \ref{ex:3space}, \ref{glynn}) of arcs which have the same size as but are not equal to normal rational curves, it is our belief that all arcs of size $q+1$ are already known, excluding the cases $k\in\{3,q-1\}$ and $q$ is even.

\begin{conjecture}\label{conject:normal rational curve}
If $6 \leq k \leq q-5$ then an arc in $\PG(k-1,q)$ of size $q+1$ is a normal rational curve.
\end{conjecture}

Observe that Theorem \ref{thm:dual_of_RS} implies the following.

\begin{theorem}\label{thm:dual_NRC}
For $k\geq 3$ and $q-k\geq 3$, if every arc of size $q+1$ in $\PG(k-1,q)$ is a normal rational curve then so is every arc of size $q+1$ in $\PG(q-k,q)$.
\end{theorem}

In the next sections we will give a proof for several instances of the MDS conjecture. 
In Section \ref{sec:state_of_the_art}, we summarize the state of the art on the MDS conjecture.

\section{The projection theorem}

Most results on arcs are, or rely on, results on planar arcs, i.e. arcs in $\PG(2,q)$. The reason for this, besides the fact that it is the smallest interesting case, is the {\em projection theorem} (Theorem \ref{thm:main_projection}), which asserts that if two distinct projections of an arc are contained in a normal rational curve then so is the arc. The fundamental ideas were developed by Segre in \cite{Segre1955b} in which the MDS conjecture is proven for $k\in \{4,5\}$ and $q$ odd.
The papers of Beniamino Segre published in the 1950's and 1960's had a huge influence on the development of combinatorial and geometric techniques which were used (first by himself and later by others) to prove several instances of the MDS conjecture. The proof of his celebrated theorem which says that an arc of size $q+1$ in $\PG(2,q)$, $q$ odd, is a conic (compare to Conjecture \ref{conject:normal rational curve}) is just one example. Segre's work was so influential that it can hardly be overestimated, and mathematics today still benefits from his efforts.

\begin{theorem}[Projection theorem]\label{thm:main_projection}
Let $\cA$ be an arc in $\PG(k-1,q)$, with $k\geq 4$ and $q\geq k+1$. If $\cA$ projects onto a subset of a normal rational curve in a hyperplane of $\PG({k-1},q)$, for two distinct points of $\cA$, then $\cA$ is contained in a normal rational curve.
\end{theorem}
\begin{proof}
We may assume that $\cA$ has size at least $k+3$, since by Lemma \ref{lem:unique_normal rational curve} every arc of size $\leq k+2$ is contained in a normal rational curve.
Suppose that each of the projections $x(\cA)$ and $y(\cA)$ from the points $x,y\in \cA$, $x\neq y$,  is contained in a normal rational curve, say $\cN_x$ and $\cN_y$.
Consider any subset $S$ of $\cA$ of size $k+2$, containing $x,y$, and let $\cN$ be the unique normal rational curve containing $S$ (here we use $q\geq k+1$). By the hypothesis, Lemma \ref{lem:unique_normal rational curve} and Lemma \ref{lem:projection_normal rational curve}, both the arc $\cA$ and the normal rational curve $\cN$ are then contained in the two distinct cones $x\cN_x$ and $y\cN_y$. 
Consider a point $z\in \cA\setminus\{x,y\}$, and let $\ell_x=xz$ and $\ell_y=yz$.
Suppose $\ell_x$ does not contain a point of $\cN$. Then $\ell_x$ is the tangent line of $\cN$ at $x$, since $x(\ell_x)$ is then the unique point of $\cN_x\setminus x(\cN)$ (see Lemma \ref{lem:extra_tangents_proj_normal rational curve}). Also $\ell_y$ is either a secant or the tangent line of $\cN$ at $y$. But then the fact that $\ell_x$ and $\ell_y$ meet contradicts Lemma \ref{lem:tangents_normal rational curve}.  
Therefore neither of the lines $\ell_x$ and $\ell_y$ are tangents. But then $z\in \cN$ since the plane $\langle \ell_x,\ell_y\rangle$ can contain at most 3 points of $\cN$. 
This shows that $\cA$ is contained in $\cN$. 
\end{proof}

\begin{corollary}\label{cor:main_projection}
Let $\cA$ be an arc in $\PG(k-1,q)$, with $k\geq 4$ and $q\geq k+1$. If the projection of $\cA$ from each $(k-3)$-subset of a $(k-2)$-subset of $\cA$ is contained in a conic then $\cA$ is contained in a normal rational curve.
\end{corollary}
\begin{proof}
The proof follows by induction on $k$ from Theorem \ref{thm:main_projection}. It coincides with Theorem \ref{thm:main_projection} in the case $k=4$. The hypothesis for $k=r$ implies the hypothesis for a projected arc (from one of the points in the $(k-2)$-subset) for $k=r-1$, from which one can apply the induction hypothesis and then Theorem \ref{thm:main_projection}.
\end{proof}

\begin{corollary}\label{cor:kaneta_maruta}
If every arc of size $q+1$ in $\PG(k-1,q)$  with $q+1\geq k+3\geq 6$ is a normal rational curve then the MDS conjecture holds true in $\PG(k,q)$.
\end{corollary}

\begin{proof}
Suppose $\cA$ is an arc in $\PG(k,q)$ of size $q+2$. Projecting $\cA$ from any of its points gives an arc of size $q+1$ in $\PG(k-1,q)$ which by hypothesis is a normal rational curve. By Theorem~\ref{thm:main_projection}, 
$\cA$ is contained in a normal rational curve, which is a contradiction since
a normal rational curve has $q+1$ points.
\end{proof}

\section{The lemma of tangents}\label{sec:lemma_of_tangents}

Let $\cA$ be an arc of $\mathrm{PG}(k-1,q)$ of size $q+k-1-t$ in which we arbitrarily order the points of $\cA$. Let $S$ be a subset of $\cA$ of size $k-2$ ordered as in $\cA$. 
Let $\alpha_1,\ldots,\alpha_t$ be $t$ linear forms whose kernels are the $t$ hyperplanes which meet $\cA$ in precisely $S$, see Lemma~\ref{triv}. 
Define, up to a scalar factor, a homogeneous polynomial of degree $t$,
\begin{eqnarray}
f_S(X)=\prod_{i=1}^t \alpha_i(X),
\end{eqnarray}
where $X=(X_1,\ldots,X_k)$. The zero locus $\cZ(f_S)$ of $f_S$ in $\PG(k-1,q)$ is called the {\em tangent hypersurface of $\cA$ at $S$}.
For any permutation $\sigma$ of the elements of $S$, define
$$
f_{\sigma(S)}(X)=(-1)^{s(\sigma)(t+1)} f_S(X),
$$
where $s(\sigma)$ is the parity of the permutation $\sigma$. 

A homogeneous polynomial $f$ in $k$ variables defines a function $\tilde{f}$ from ${\mathbb F}_q^k$ to ${\mathbb F}_q$ under evaluation. If we change the basis of ${\mathbb F}_q^k$ then although the polynomial $f$ will change, the function $\tilde{f}$ will not. Put another way, any function from ${\mathbb F}_q^k$ to ${\mathbb F}_q$ is the evaluation of a polynomial once we fix a basis of ${\mathbb F}_q^k$. Obviously, the polynomial we obtain depends on the basis we choose.

The following is the coordinate-free version of Segre's lemma of tangents.

\begin{lemma}[Coordinate-free lemma of tangents] \label{segrelemma} 
Let $\cA$ be an arc of $\PG(k-1,q)$ and let $D$ be a subset of $\cA$ of size $k-3$. For all $x,y,z \in \cA \setminus D$,
$$
f_{D \cup \{x\}}(y)f_{D \cup \{y\}}(z)f_{D \cup \{z\}}(x)$$
$$=(-1)^{t+1} f_{D \cup \{y\}}(x)f_{D \cup \{z\}}(y)f_{D \cup \{x\}}(z).
$$
\end{lemma}

\begin{proof} 

Let $f_a^*$ be the polynomial we obtain from $f_{D \cup \{a \}}$ with respect to the basis $B= \{x,y,z \} \cup D $.

Then 
$$f_x^*(X)=\prod_{i=1}^t (a_{i2} X_2+a_{i3}X_3),~f_y^*(X)=\prod_{i=1}^t (b_{i1} X_1+b_{i3}X_3),~f_z^*(X)=\prod_{i=1}^t (c_{i1} X_1+c_{i2}X_2),
$$
for some $a_{ij},b_{ij},c_{ij} \in {\mathbb F}_q$.
Let $s \in \cA \setminus B$. The hyperplane joining $s$ and $D\cup\{x\}$ is $\ker (s_3X_2-s_2X_3)$ where $(s_1,s_2,\ldots,s_k)$ are the coordinates of $s$ with respect to the basis $B$.

As $s$ runs through the elements of $\cA \setminus B$, the element $-s_2/s_3$ runs through the elements of ${\mathbb F}_q\setminus \{ a_{i3}/a_{i2} \ | \ i=1,\ldots,t \}$, the latter being excluded since the tangent hypersurface of $\cA$ at $D\cup\{x\}$ does not pass
through $s$. Using the fact that the product of all the non-zero elements of ${\mathbb F}_q$ is $-1$, we obtain
$$
\prod_{s \in \cA \setminus B} \frac{-s_2}{s_3} \prod_{i=1}^t \frac{a_{i3}}{a_{i2}}=-1,
$$
and since $\prod_{i=1}^t a_{i3}=f_x^*(z)$ and $\prod_{i=1}^t a_{i2}=f_x^*(y)$, we have
$$
f_x^*(z)\prod_{s \in \cA \setminus B} (-s_2) =f_x^*(y) \prod_{s \in \cA \setminus B} s_3.
$$
This equation was obtained by considering all hyperplanes through $D\cup\{x\}$. Similarly, by considering the hyperplanes through $D\cup\{y\}$, we get
$$
f_y^*(x)\prod_{s \in \cA \setminus B} (-s_3) =  f_y^*(z) \prod_{s \in \cA \setminus B} s_1
$$
and by considering the hyperplanes through $D\cup\{z\}$ we get
$$
f_z^*(y)\prod_{s \in \cA \setminus B} (-s_1) = f_z^*(x)\prod_{s \in \cA \setminus B} s_2.
$$
Combining these three equations, we obtain
$$
f_x^*(z)f_y^*(x)f_z^*(y)=(-1)^{t+1} f_x^*(y)f_y^*(z)f_z^*(x).
$$
Now, since $f^*$ and $f$ define the same functions on the points of $\PG(k-1,q)$, the lemma follows.
\end{proof}


Let $\mathcal E$ be the set of the first $k-2$ elements of $\cA$. For each $(k-2)$-subset $S\subset \cA$,
scale the polynomial $f_S(X)$ so that
\begin{eqnarray}\label{eqn:scaling}
f_S(e)=(-1)^{s(t+1)}f_{S\cup\{e\}\setminus \{a\}}(a),
\end{eqnarray}
where $e$ is the first element of $\mathcal E\setminus S$, $a$ is the last element of $S\setminus \mathcal E$, and $s$ is the parity of the permutation which orders $S\cup \{e\}$ as in the ordering of $\cA$ (to determine the value of $s$ we assume the ordering of $\cA$ for the subset $S$). With this notation it should be understood that  the order is respected when taking the union of ordered sets, i.e. with ``union" we mean the concatenation of the ordered sets.

So that this definition makes sense, we fix a vector representative for each point of $\cA$. This is essential, since the scaling will depend on which vector representative we choose for a particular point.

\begin{lemma} \label{segrelemmascaled} 
Let $\cA$ be an arc of $\PG(k-1,q)$ and let $D$ be a subset of $\cA$ of size $k-3$. For all $x,y,z \in \cA \setminus D$,
$$
f_{D \cup \{x\}}(y)=(-1)^{s(\sigma)(t+1)} f_{D \cup \{y\}}(x),
$$
where $\sigma$ is the permutation that orders $(D \cup \{ x\},y)$ as $(D \cup \{ y\},x)$ and $s(\sigma)$ is the parity of $\sigma$.
\end{lemma}

\begin{proof}
We will prove this inductively on $|D \cap \mathcal E|$. Suppose that $|D \cap \mathcal E|=k-3$ and let $e$ be the unique element of $\mathcal E \setminus D$. Applying Lemma~\ref{segrelemma} to $D$ and $x,y,e$, we have that
$$
f_{D \cup \{x\}}(y)f_{D \cup \{y\}}(e)f_{D \cup \{e\}}(x)=(-1)^{t+1} f_{D \cup \{y\}}(x)f_{D \cup \{e\}}(y)f_{D \cup \{x\}}(e).
$$
Since we scaled $f_{D \cup \{x\}}(X)$ so that 
$$
f_{D \cup \{x\}}(e)=(-1)^{s(\theta)(t+1)}f_{D \cup \{e \}}(x), 
$$
where $\theta$ is the permutation that orders $(D \cup \{ x\},e)$ as $(D \cup \{ e\},x)$,
and we scaled $f_{D \cup \{y\}}(X)$ so that 
$$
f_{D \cup \{y\}}(e)=(-1)^{s(\tau)(t+1)}f_{D \cup \{e \}}(y), 
$$
where $\tau$ is the permutation that orders $(D \cup \{ y\},e)$ as $(D \cup \{ e\},y)$, we have that
$$
f_{D \cup \{x\}}(y)=(-1)^{(s(\theta)+s(\tau)+1)(t+1)}f_{D \cup \{y \}}(x). 
$$
It suffices to observe now that $s(\sigma)=s(\theta)+s(\tau)+1$, where $\sigma$ is the permutation that orders $(D \cup \{ x\},y)$ as $(D \cup \{ y\},x)$.

Now suppose that $|D \cap \mathcal E| \leq k-4$ and that the statement holds for $D$ such that $|D \cap \mathcal E|$ is larger. Let $e$ be the first element of $ \mathcal E\setminus (D \cup \{y\})$. 

If $e=x$ then we are done since we scaled $f_{D \cup \{y\}}(X)$ so that 
$$
f_{D \cup \{y\}}(e)=(-1)^{s(\sigma)(t+1)}f_{D \cup \{e \}}(y), 
$$
where $\sigma$ is the permutation that orders $(D \cup \{ y\},e)$ as $(D \cup \{ e\},y)$.

If $e \neq x$ then let $z$ denote the last element of $D \cup \{y\}$. By scaling,
$$
f_{D \cup \{y\}}(e)=(-1)^{s(\tau)(t+1)}f_{D \cup \{e,y \}\setminus \{z\}}(z),
$$
where $\tau$ is the permutation that orders $(D \cup \{ y\},e)$ as $(D \cup \{ e,y\}\setminus \{z\},z)$. By induction, 
$$
f_{D \cup \{e,y \}\setminus \{z\}}(z)= (-1)^{s(\theta)(t+1)}f_{D \cup \{e\}}(y),
$$
where $\theta \circ \tau$ is the permutation that orders $(D \cup \{ y\},e)$ as $(D \cup \{ e\},y)$.

Observe that $e$ is also the first element of $\mathcal E \setminus (D \cup \{x\})$, so we also have that
$$
f_{D \cup \{x\}}(e)=(-1)^{s(\alpha)(t+1)}f_{D \cup \{e\}}(x),
$$
where $\alpha$ is the permutation that orders $(D \cup \{ x\},e)$ as $(D \cup \{ e\},x)$.

Now, applying Lemma~\ref{segrelemma}, to $D$ and $x,y,e$ the lemma follows observing again that if $\sigma$ is the permutation that orders $(D \cup \{ y\},x)$ as $(D \cup \{ x\},y)$ then 
$$
s(\sigma)=s(\theta)+s(\tau)+s(\alpha)+1.
$$
This concludes the proof.
\end{proof}

Define the function $g$ on ordered subsets of $\cA$ of size $k-1$ by
\begin{eqnarray}\label{eqn:g}
g(S\cup\{a\})= (-1)^{s(t+1)}f_S(a),
\end{eqnarray} 
where $S$ is an ordered subset of $\cA$ of size $k-2$ and $s$ is the parity of the permutation which orders $S\cup \{a\}$ as in the ordering of $\cA$. Note that $S$ is considered as an unordered set in the notation $f_S(a)$.

\begin{lemma}[Scaled coordinate-free lemma of tangents]\label{lem:g_functionpre}
If $\sigma$ is a permutation in $\mathrm{Sym}(k-1)$ and $C$ is an ordered $(k-1)$-subset of $\cA$ then
$$
g(C^\sigma)=(-1)^{s(t+1)}g(C),
$$
where $s$ is the parity of the permutation $\sigma$.
\end{lemma}

\begin{proof}
This follows from Lemma~\ref{segrelemmascaled} and the observation that  
$$
(-1)^{s(\sigma \circ \tau)}=(-1)^{s(\sigma)+s( \tau)}
$$
for permutations $\sigma$ and $\tau$.
\end{proof}


\section{A system of equations associated with an arc} \label{systemsection}

Let $\cA$ be an arc of $\PG(k-1,q)$ of size $q+k-1-t \geq k+t$ arbitrarily ordered and let $E$ be a subset of $\cA$ of size $k+t$.
For a $(k-1)$-subset $C$ of $\cA$, we define $g(C)$ as in the previous section, see (\ref{eqn:g}). 
We write
$$
\det(u,C)
$$
to mean the determinant of the matrix whose first row is $u$ and whose remaining $k-1$ rows are the vectors in $C$ in the order in which they are ordered in $\cA$. Other expressions for $\det$ are similarly defined and $X$ is understood to mean the indeterminants $(X_1,\ldots,X_k)$.

\begin{lemma} \label{sumeqnthm}
For any subset $S$ of $E$ of size $k-2$,
$$
\sum_{C} g(C) \prod_{u \in E \setminus C} \det(u,C)^{-1}=0,
$$
where the sum is over $(k-1)$-subsets $C$ of $E$ containing $S$.
\end{lemma}

\begin{proof}
Let $S$ be a subset of $E$ of size $k-2$. 
Since there are $t+1$ determinants in the product, any permutation $\sigma$ of $C$ will change the sign by $(-1)^{s(\sigma)(t+1)}$ which, by Lemma~\ref{lem:g_functionpre}, will be cancelled out by the same sign appearing due to the permutation of $C$ in the $g(C)$ term. Thus, the summand in the equation is unaffected by the ordering of $C$.
Hence, it suffices to prove the equation for ordered subsets $C$ of $E$ in which the first $k-2$ elements are $S$.

Let $B$ be a subset of $\cA$ of size $k$ containing $S$. With respect to the basis $B$, the polynomial $f_S(X)$ is a homogeneous polynomial in two variables of degree $t$. 

Let $x$ be a fixed element of $E \setminus S$. The polynomial
$$
\sum_{e \in E\setminus (S \cup \{x \})} f_S(e) \prod_{ u \in  E\setminus (S \cup \{x,e \})} \frac{\det(X,u,S)}{\det(e,u,S)}
$$ 
evaluates to $f_S(e)$ for all $e \in E\setminus (S \cup \{x \})$. Since this polynomial and $f_S(X)$ are both homogeneous polynomials in two variables of degree $t$ which agree at $t+1$ distinct points, they are the same. Evaluating the polynomial in $X=x$, we have
$$
f_S(x)=\sum_{e \in E\setminus (S \cup \{x \})} f_S(e) \prod_{ u \in  E\setminus (S \cup \{x,e \})} \frac{\det(x,u,S)}{\det(e,u,S)}.
$$ 
Dividing by  
$$
\prod_{ u \in  E\setminus (S\cup \{x \})} \det(x,u,S)
$$ 
gives
$$
f_S(x)\prod_{ u \in  E\setminus (S \cup \{x \})} \det(x,u,S)^{-1}=-\sum_{e \in E\setminus (S \cup \{x \})} f_S(e) \prod_{ u \in  E\setminus (S \cup \{e \})} \det(e,u,S)^{-1},
$$
from which we deduce that
$$
\sum_{e \in E\setminus S} f_S(e) \prod_{ u \in  E\setminus (S \cup \{e \})} \det(u,S,e)^{-1}=0,
$$
and the theorem follows.
\end{proof}

If $k \leq p$ then the system of equations in Lemma~\ref{sumeqnthm} contains a set of equations which will be put to use in the classification of large arcs in this range and verifies the MDS conjecture for these values of $k$.

\begin{lemma} \label{eqnsksmall}
Let $E$ be a subset of $\cA$ of size $k+t$ and let $\Delta$ be a subset of $E$ of size $t+2$. If $k \leq p$ then
$$
\sum_C g(C) \prod_{u \in E \setminus C} \det(u,C)^{-1}=0,
$$
where the sum runs over the $(k-1)$-subsets of $E$ contained in $\Delta$.
\end{lemma}

\begin{proof}
For any subset $S$ of $E$ of size $k-2$, Lemma~\ref{sumeqnthm} implies that
\begin{eqnarray}\label{eqn:equation_S}
\mathrm{eqn}(S):=\sum_C g(C) \prod_{u \in E \setminus C} \det(u,C)^{-1}=0,
\end{eqnarray}
where the sum runs over the $(k-1)$-subsets of $E$ containing $S$.
For each such $S$, define
$$
\lambda_S=(-1)^m(k-m-2)!m!
$$ 
where $m=|S \cap \Delta|$, and consider the sum
\begin{equation} \label{Ssum}
\sum_S \lambda_S~ \mathrm{eqn}(S),
\end{equation}
where $S$ runs over all $(k-2)$-subsets of $E$.

Let $C$ be a $(k-1)$-subset of $E$ and let $m=|C \cap \Delta|$. 

If $C \not\subseteq \Delta$ then the term 
$$
g(C) \prod_{u \in E \setminus C} \det(u,C)^{-1}
$$
appears in the sum (\ref{Ssum}) with coefficient
$$
(k-1-m)(-1)^m(k-m-2)!m!+m(-1)^{m-1}(k-m-1)!(m-1)!=0,
$$
where the first term comes from the subsets $S$ of $C$ for which $|S \cap \Delta|=m$ and the second term comes from the subsets $S$ of $C$ for which $|S \cap \Delta|=m-1$.

If $C \subseteq \Delta$ then the term 
$$
g(C) \prod_{u \in E \setminus C} \det(u,C)^{-1}
$$
appears with coefficient
$$
(-1)^{k-2}(k-1)!
$$
from which the assertion follows.
\end{proof}


\section{A proof of the MDS conjecture for prime fields}\label{sec:MDS_p}

The following theorem verifies the MDS conjecture for prime fields. We suppose throughout that $q=p^h$, for some prime $p$.

\begin{theorem} \label{MDSprime}
If $4\leq k \leq p$ then an arc of $\mathrm{PG}(k-1,q)$ has at most $q+1$ points.
\end{theorem}

\begin{proof}
Let $\cA$ be an arc of $\PG(k-1,q)$ of size $q+2=q+k-1-t$, so $t=k-3$, and define $g(C)$ as in (\ref{eqn:g}). By Theorem~\ref{dualarc}, we can assume that $k \leq \frac{1}{2}q+1$, so $k+t=2k-3 \leq q-1$. Let $E$ be a subset of $\cA$ of size $k+t=2k-3$ and let $\Delta$ be a subset of $E$ of size $t+2=k-1$. By Lemma~\ref{eqnsksmall},
$$
g(\Delta) \prod_{u \in E \setminus \Delta} \det(u,\Delta)^{-1}=0,
$$
which is a contradiction, since $g(\Delta) \neq 0$ and $\det(u,\Delta) \neq 0$, for all $u \in \cA \setminus \Delta$.
\end{proof}


\section{A proof of the MDS conjecture for $k \leq 2p-2$}\label{sec:MDS_2p-2}

In a similar vein to Theorem~\ref{MDSprime} we will prove that the MDS conjecture is true for $k \leq 2p-2$.

Let $\cA$ be an arc of $\PG(k-1,q)$ of size $q+2$ and define $g(C)$ as in (\ref{eqn:g}).
Let $X=\{x_1,\ldots,x_{k-1}\}$, $Y=\{ y_1,\ldots,y_{k-2} \}$ be disjoint subsets of $\cA$ and define $E=X \cup Y$. Note that $X$ is a set in this section and not a vector of indeterminants as in Section~\ref{systemsection}.
Let $r$ be a non-negative integer such that $r \leq k-p$ and define 
$X_r=\{x_1,\ldots,x_{r}\}$ and $Y_r=\{ y_1,\ldots,y_{r} \}$. 
For any $T \subseteq \{1,2,\ldots,r\}$, let 
$$
X_{T}= \{ x_j \ | \ j \in T \},\ \ Y_{T}= \{ y_j \ | \ j \in T \}.
$$
Define a $(k-1)$-subset $C_{r,T}$ of $E$ by
$$
C_{r,T}=(X \setminus X_r) \cup X_{T} \cup (Y_r\setminus Y_T).
$$

\begin{lemma} \label{twotothed}
Let $\cA$ be an arc of $\PG(k-1,q)$ of size $q+2$ and suppose $k \geq p$. Let $E=X \cup Y$ be a subset of $\cA$ of size $2k-3$ defined as above. Then
$$
\sum_{T \subseteq \{1,\ldots,k-p\}} (-1)^{|T|} g(C_{k-p,T}) \prod_{u \in E \setminus C_{k-p,T}} \det(u,C_{k-p,T})^{-1}=0.
$$
\end{lemma}

\begin{proof}
Let $S$ be a subset of $E$ of size $k-2$ and set $r=k-p$.

If $S \cap (X_r \cup Y_r)=X_{T} \cup (Y_r \setminus Y_{T})$ for some $T \subseteq \{1,\ldots,r\}$ then let 
$$
\lambda_S=m!(p-2-m)!(-1)^{m+|T|},
$$
where $m=|S \cap(X \setminus X_r)|$, otherwise put $\lambda_S=0$. 

Let $ \alpha_C$ denote the coefficient of the $C$-term in the sum 
\begin{equation} \label{tausum}
\sum_{S\subset E} \lambda_S \mathrm{eqn}(S).
\end{equation}
Then $\alpha_C$ is zero if $|C\cap (X_r\cup Y_r)|\notin \{r,r+1\}$.

If $|C\cap (X_r\cup Y_r)|=r+1$ then there exists some $i\in \{1,\ldots, r\}$ for which $x_i,y_i\in C$. If $i$ is not unique then $\lambda_S=0$ for all $(k-1)$-subsets $S$ of $C$, so $\alpha_C=0$. If $i$ is unique then
$$
\alpha_C=\lambda_{C \setminus \{x_i\}}+\lambda_{C \setminus \{y_i\}}=0.
$$ 

If $|C\cap (X_r\cup Y_r)|=r$ and $C\cap (X_r \cap Y_r)\neq X_T \cup (Y \setminus Y_T)$, for any $T\subset\{1,\ldots,r\}$, then $\alpha_C=0$.

If $|C\cap (X_r\cup Y_r)|=r$ and $C\cap (X_r \cap Y_r)=X_T \cup (Y \setminus Y_T)$, for some $T\subset\{1,\ldots,r\}$, and $C \setminus (X_r \cap Y_r) \neq X \setminus X_r$ then 
$$
\alpha_C=(-1)^{|T|}(m(m-1)!(p-1-m)!(-1)^{m-1}+(p-1-m)m!(p-2-m)!(-1)^s)=0
$$
where $m=|C\cap (X\setminus X_r)|$.

Finally, if none of the above then $C=C_{r,T}$ for some $T\subset\{1,\ldots,r\}$ and
$$
\alpha_C=(-1)^{|T|}(k-1-r)(p-2)!=(-1)^{|T|}(p-1)!
$$
which is nonzero.
\end{proof}

\begin{lemma} \label{interpaform}
Let $C$ be a $(k-1)$-subset of  $\cA$. Suppose $W$ and $\Delta$ are disjoint subsets of $\cA$ such that $|W|=k-|\Delta|$, and $\Delta \subseteq C$. Then
$$
\sum_{w \in W}\frac{\det(u,W\setminus w,\Delta)}{\det(w,W\setminus w,\Delta)} \det(w,C)=\det(u,C).
$$
\end{lemma}

\begin{proof}
This is interpolation of the linear form $\det(u,C)$ in $u=(u_1,\ldots,u_k)$ at the points of $\Delta \cup W$. 
\end{proof}

\begin{theorem}
If $q$ is not prime and $k \leq 2p-2$ then an arc of $\mathrm{PG}(k-1,q)$ has at most $q+1$ points.
\end{theorem}

\begin{proof}
Let $\cA$ be an arc of $\PG(k-1,q)$ of size $q+2$. By Theorem~\ref{MDSprime}, we can assume that $k \geq p$.
Let $r$ be a non-negative integer such that $r \leq k-p$. Define $g(C)$ as in (\ref{eqn:g}) and define $X$, $Y$, $E$, $X_r$, $Y_r$, $X_{T}$, $Y_{T}$ and $C_{r,T}$ as before.

We will prove that 
$$
\sum_{T \subseteq \{1,\ldots,r\}} (-1)^{|T|} g(C_{r,T}) \prod_{u \in E \setminus C_{r,T}} \det(u,C_{r,T})^{-1} =0
$$
implies that
$$
\sum_{T \subseteq \{1,\ldots,r-1\}} (-1)^{|T|} g(C_{r-1,T}) \prod_{u \in E \setminus C_{r-1,T}} \det(u,C_{r-1,T})^{-1} =0.
$$
By iteration, this implies that
$$
g(X)\prod_{u \in Y} \det(u,X)^{-1}=0,
$$
which is a contradiction.

By Lemma~\ref{twotothed}, the statement holds for $r= k-p$, so it suffices to show the inductive step.

Let $W$ and $\Delta$ be disjoint subsets of $\cA \setminus \{x_1,\ldots,x_{k-1},y_1,\ldots,y_r\}$ of size $r+1$ and $k-1-r$ respectively.
Let $\Omega=W \cup \{y_{2r+1},\ldots,y_{k-2}\}$. 
Note that $k-2 \geq 2r$, since $r \leq k-p$ and $k \leq 2p-2$, so $|\Omega|=k-1-r$.
Let $w \in W$. By induction, replacing $\{y_{r+1},\ldots,y_{2r}\}$ by $W \setminus \{w\}$,
$$
\sum_{T \subseteq \{1,\ldots,r\}} (-1)^{|T|} g(C_{r,T}) \prod_{u \in (X \cup Y_r \cup \Omega) \setminus C_{r,T}} \det(u,C_{r,T})^{-1} \det(w,C_{T})=0.
$$
Hence,
$$
\sum_{T \subseteq \{1,\ldots,r\}} (-1)^{|T|} g(C_{r,T}) \prod_{u \in (X \cup Y_r \cup \Omega) \setminus C_{r,T}} \det(u,C_{r,T})^{-1} 
\frac{\det(y_r,W\setminus w,\Delta)}{\det(w,W\setminus w,\Delta)} \det(w,C_{r,T})=0.
$$
Summing over $w \in W$, Lemma ~\ref{interpaform} implies
$$
\sum_{T \subseteq \{1,\ldots,r\}} (-1)^{|T|} g(C_{r,T}) \prod_{u \in (X \cup Y_r \cup \Omega)  \setminus C_{r,T}} \det(u,C_{r,T})^{-1} \det(y_r,C_{r,T})=0.
$$
Since $\det(y_r,C_{r,T})=0$ if $r \not\in T$, this implies that
$$
\sum_{T \subseteq \{1,\ldots,r-1\}} (-1)^{|T|} g(C_{r-1,T}) \prod_{u \in (X \cup Y_{r-1} \cup \Omega)\setminus C_{r-1,T}} \det(u,C_{r-1,T})^{-1} =0.
$$
This equation no longer depends on $y_r$, so we could have chosen $y_r$ to be an arbitrary point of $\cA$ and chosen
$$
\Omega=\{y_{r},\ldots,y_{k-2}\},
$$
which proves the inductive step.
\end{proof}


\section{The classification of the largest arcs for $k \leq p$ and $k\neq \frac{1}{2}(q+1)$} 

The proof of the following theorem exploits the system of equations in Lemma~\ref{eqnsksmall}.

\begin{theorem} \label{itsannormal rational curve}
If $k \leq p$ and $k \neq \frac{1}{2}(q+1)$ then an arc of $\mathrm{PG}(k-1,q)$ of size $q+1$ is a normal rational curve.
\end{theorem}

\begin{proof}
Let $\cA$ be an arc of size $q+1$. By Corollary~\ref{dualarc}, we may assume that $k \leq \frac{1}{2}(q+1)$. Let $E$ be a subset of $\cA$ of size $k+t=2k-2$.
Let $\Delta$ be a subset of $E$ of size $t+2=k$.

By Lemma~\ref{eqnsksmall},
$$
\sum_C g(C) \prod_{u \in E \setminus C} \det(u,C)^{-1}=0,
$$
where the sum runs over the $(k-1)$-subsets $C$ of $E$ contained in $\Delta$.

Let $W$ be a subset of $\cA \setminus \Delta$ of size $k-2$ and let $w\in W$. Let $x \in \cA \setminus (\Delta \cup W)$ and apply the above equation with $E=\Delta \cup (W \setminus \{ w\}) \cup \{x\}$. Then
$$
\sum_C g(C) \prod_{u \in W} \det(u,C)^{-1} \det(\Delta \setminus C,C)^{-1} \det(x,C)^{-1}\det(w,C)=0,
$$
where the sum runs over the $(k-1)$-subsets of $E$ contained in $\Delta$.

Let $D$ be a subset of $\Delta$ of size $3$. Since $|D|=3$ and $|W|=k-2$, there is a linear combination $v$ of the vectors $\{ w \ | \ w \in W\}$ such that $v$ is in the subspace spanned by $D$. Hence, $\det(v,C)=0$ for all $(k-1)$-subsets $C$ for which $D \subset C$. Thus, we can sum the above equations, so that
$$
\sum_C g(C) \prod_{u \in W} \det(u,C)^{-1} \det(\Delta \setminus C,C)^{-1} \det(x,C)^{-1}\det(v,C)=0,
$$
where the sum is over the three $(k-1)$-subsets of $\Delta$ which do not contain $D$.

With respect to the basis $\Delta=\{e_1,\ldots,e_k\}$, ordered so that the elements of $D$ are the first three elements of the basis, the above sum is
$$
\frac{a_1}{x_1}+\frac{a_2}{x_2}+\frac{a_3}{x_3}=0,
$$
where
$$
a_i=g(\Delta \setminus \{e_i\}) \prod_{u \in W}\det(u,\Delta \setminus \{e_i\})^{-1} \det(e_i,\Delta \setminus \{e_i\})^{-1} \det(v,\Delta \setminus \{e_i\}).
$$
Thus we have that for all $x \in \cA \setminus W $,
$$
a_1x_2x_3+a_2x_1x_3+a_3x_1x_2=0.
$$
The coefficients $a_1$, $a_2$ and $a_3$ are determined, up to scalar factor, by two points of $\cA \setminus (W \cup \Delta)$. Thus, since
$$
|\cA|-(2k-2) \geq 3,
$$
we can switch an element of $W$ with an element of $\cA \setminus (W \cup \Delta)$ and conclude that
$$
a_1x_2x_3+a_2x_1x_3+a_3x_1x_2=0,
$$
for all $x \in \cA$.

This implies that the projection of $\cA$ from $\Delta \setminus D$ is contained in a conic, for any subsets $\Delta$ and $D$ of $\cA$. By Corollary~\ref{cor:main_projection}, $\cA$ is a normal rational curve.
\end{proof}


\section{Extending small arcs to large arcs}

Let $\cG$ be an arc of $\PG(k-1,q)$ arbitrarily ordered.
Suppose that $\cG$ can be extended to an arc $\cA$ of $\PG(k-1,q)$ of size $q+k-1-t \geq k+t$.
Let $n=|\cG|-k-t$ be a non-negative integer.

Let $\mathrm{P}_n$ be the matrix whose columns are indexed by the $(k-1)$-subsets $C$ of $\cG$ and whose rows are indexed by pairs $(S,U)$, where $S$ is a $(k-2)$-subset of $\cG$ and $U$ is a $n$-subset of $\cG \setminus S$. The $((S,U),C)$ entry of $\mathrm{P}_n$ is zero unless $C$ contains $S$ in which case it is $ \prod_{ u \in U} \det(u,C)$.

The following theorem is from \cite{Chowdhury2015}, see also \cite{Ball2018}.

\begin{theorem}
If an arc $\cG$ of $\PG(k-1,q)$ can be extended to an arc of size $q+2k-1-|\cG|+n$ then the system of equations $\mathrm{P}_n v=0$ has a solution in which all the coordinates of $v$ are non-zero.
\end{theorem}

\begin{proof}
Let $|\cG|=k+t+n$ and suppose that $\cG$ extends to an arc $\cA$ of size $q+k-1-t$. 

Let $U$ be a subset of $\cG$ of size $n$. Then $E=\cG \setminus U$ is a subset of $\cG$ of size $k+t$. By Lemma~\ref{sumeqnthm}, for each subset $S$ of $E$ of size $k-2$,
$$
\sum_{C} g(C) \prod_{u \in \cG \setminus (U \cup C)} \det(u,C)^{-1} =0,
$$
where the sum runs over all $(k-2)$-subsets $C$ of $E$ containing $S$.

Then
$$
\sum_{C} \prod_{u \in U} \det(u,C) \left(g(C) \prod_{u \in G \setminus C} \det(u,C)^{-1} \right)=0.
$$
This system of equations is given by the matrix $\mathrm{P}_n$ and a solution $v$ is a vector with $C$ coordinate 
$$
\alpha_{C,\cG}=g(C) \prod_{u \in \cG \setminus C} \det(u,C)^{-1} 
$$ 
which is non-zero for all $(k-1)$-subsets $C$ of $\cG$.
\end{proof}

Suppose that we do find a solution $v$ to the system of equations. Then we know the value of $\alpha_{C,\cG}$ and therefore $f_{S} (x)$, where $C=S \cup \{ x\}$. This would allow one to calcuate the polynomials $f_S(X)$ for each subset $S$ of $\cG$ of size $k-2$. Therefore, if $\cG$ does extend to an arc $\cA$ then each solution tells us precisely the tangent hyperplanes to $\cA$ containing $S$, for each $(k-2)$-subset $S$ of $\cG$.

By starting with a sub-arc $\cG$ of size $3k-6$ of the normal rational curve one can prove that the matrix $\mathrm{P}_n$ has full rank and conclude the following theorem, see \cite{BdB2017}.

\begin{theorem}
If $\cG$ is a subset of the normal rational curve of $\PG(k-1,q)$ of size $3k-6$ and $q$ is odd, then $\cG$ cannot be extended to an arc of size $q+2$. 
\end{theorem}

\section{The Segre-Blokhuis-Bruen-Thas hypersurface}

The {\em Segre-Blokhuis-Bruen-Thas hypersurface} associated to an arc $\cA$ in $\PG(k-1,q)$ is a hypersurface in the dual of $\PG(k-1,q)$ containing the set of points which are dual to the hyperplanes of $\PG(k-1,q)$ meeting $\cA$ in $k-2$ points. Its existence was proved in a sequence of papers by Segre (\cite{Segre1967}) and Blokhuis, Bruen and Thas (\cite{BBT1988} and \cite{BBT1990}). The following theorem gives an alternative proof. It relies on the scaled coordinate-free lemma of tangents (Lemma~\ref{lem:g_functionpre}). 

For $i\in \{1,\ldots,k-1\}$ denote by $X_i$ the $(k-1)$-tuple $(X_{i1},X_{i2},\ldots,X_{ik-1})$ and by $\det_j(X_1,\ldots,X_{k-1})$ the determinant in which the $j$-coordinate of each of the $X_i$'s has been deleted.

\begin{theorem} \label{BBTthm}
Let $m\in \{1,2 \}$ such that $m-1=q$ modulo $2$. If $\cA$ is an arc of $\mathrm{PG}(k-1,q)$ of size $q+k-1-t$, where $|\cA| \geq mt+k-1$, then there is a homogeneous polynomial in $k$ variables $\phi(Z)$, of degree $mt$, which gives a polynomial $G(X_1,\ldots,X_{k-1})$ under the substitution $Z_j=\det_j(X_1,\ldots,X_{k-1})$, with the property that for all $\{ y_1,\ldots,y_{k-2} \}\subset \cA$
$$
G(X,y_1,\ldots,y_{k-2})=f_{\{y_1,\ldots,y_{k-2}\}}(X)^m.
$$ 
Moreover, $\phi(Z)$ is unique.
\end{theorem}

\begin{proof}
Order the arc $\cA$ arbitrarily and let $E$ be a subset of $\cA$ of size $mt+k-1$. Define
\begin{eqnarray}\label{eqn:G}
G(X_1,\ldots,X_{k-1})=\sum_{T} \left ( f_{T\setminus \{a_{k-1}\}}(a_{k-1})\right )^m \prod_{u \in E \setminus T} \frac{\det(X_1,\ldots,X_{k-1},u)}{\det(a_1,\ldots,a_{k-1},u)}.
\end{eqnarray}
where the sum runs over subsets $T=\{a_1,\ldots,a_{k-1}\}$ of $E$.

Observe that $G$ can be obtained from a homogeneous polynomial of degree $mt$ in $Z_1,\ldots,Z_k$ under the change of variables
$Z_j=\det_j(X_1,\ldots,X_{k-1})$.

For $S=\{y_1,\ldots,y_{k-2}\}$, define
$$h_S(X):=G(y_1,\ldots,y_{k-2},X).$$

Note that $h_S(X)$ is well-defined since any reordering of $S$ can only ever multiply $h_S(X)$ by $(-1)^{mt}=1$.

For $S\subset E$, the only nonzero terms in $h_S(X)$ are obtained for subsets $T$ of $E$ containing
$S$. Therefore,
$$
h_S(X)=\sum_{a\in E\setminus S} \left ( f_{S}(a)\right )^m \prod_{u \in E \setminus (S\cup\{a\})} \frac{\det(y_1\ldots,y_{k-2},X,u)}{\det(y_1,\ldots,y_{k-2},a,u)}.
$$
The evaluation of $h_S(X)$ at $x\in E$ is equal to zero if $x\in S$ and equal to $(f_S(x))^m\neq 0$ otherwise. Since, with respect to a basis containing $S$ both $f_S^m$ and $h_S$ are homogeneous polynomials in two variables of degree $mt$ and $E\setminus S$ has size $mt+1$, we may conclude that $h_S=f_S^m$.

If $S$ is not contained in $E$ then we proceed by induction on the size of $S\setminus E$. As induction hypothesis we assume that for each subset $S$ with $S\setminus E$ of size $r$ the polynomials $h_S$ and $f_S^m$ are equal. Let $S=\{y_1,\ldots,y_{k-2}\}$ be such that $S\setminus E$ is of size $r+1$. W.l.o.g. assume $y_{k-1}\notin E$. Then for $x\in E$ we have
$$
h_S(x)=(-1)^{mt}h_{S'}(y_{k-1})=h_{S'}(y_{k-1}),
$$
where $S'$ is the set obtained from $S$ by replacing the $(k-1)$-th element $y_{k-1}$ of $S$ by $x$.
On the other hand, by the definition (\ref{eqn:g}) of $g$ and Lemma~\ref{lem:g_functionpre}, the scaled coordinate-free lemma of tangents, we have
$$
(f_S(x))^m=(g(y_1,\ldots,y_{k-1},x))^m=(g(y_1,\ldots ,y_{k-2},x,y_{k-1}))^m=(f_{S'}(y_{k-1}))^m.
$$
By induction, $h_{S'}(y_{k-1})=(f_{S'}(y_{k-1}))^m$, and therefore the polynomials $h_S$ and $f_S^m$ have the same evaluation at points in $E$. Applying the same argument as in the case where $S\subset E$ we obtain $h_S=f_S^m$.

Denote by $\phi(Z)$ the polynomial obtained from $G(X_1,\ldots,X_{k-1})$
under the substitution 
$$Z_j=\det_j(X_1,\ldots,X_{k-1}),$$
$j=1,\ldots,k$.
Then $\phi(Z)$ has degree $mt$. Consider any $S=\{y_1,\ldots,y_{k-2}\} \subset \cA$ and any point $x$ contained in the tangent hypersurface $\cZ(f_S)$ of $\cA$ at $S$.
Then $G(x,y_1,\ldots,y_{k-2})=f_S(x)^m=0$. 
This implies that $\phi(z)=0$, where $z=(z_1,\ldots,z_k)$
and
$$
z_j=\det_j(x,y_1,\ldots,y_{k-2}).
$$
Therefore, the polynomial $\phi(Z)$ vanishes at the set of points, denoted by $\mathcal T$, of the dual space which are dual to the hyperplanes meeting $\cA$ in $k-2$ points.
It remains to prove that $\phi$ is unique. Suppose that $\phi(Z)$ and $\phi'(Z)$ are two forms of degree $mt$ vanishing on $\mathcal T$.

We will prove by induction on $r$ that for a subset $D$ of $\cA$ of size $k-r$, there is a constant $a \in {\mathbb F}_q$, such that the restriction of $\phi-a\phi'$ to $\langle D\rangle^{\perp}$ is zero. Once this is established uniqueness follows when $r=k$.

For $r=2$, $\langle D\rangle^{\perp}$ is a line containing $t$ points of $\mathcal T$. Both $\phi$ and $\phi'$ have zeros of multiplicity $m$ at these $t$ points. Thus, selecting $a$ so that $\phi-a\phi'$ is zero at a point distinct from these $t$ points, the polynomial 
$\phi-a\phi'$ is zero at $mt+1$ points (counting with multiplicity) of the line $\langle D\rangle^{\perp}$, which implies it is zero on all the points of $\langle D\rangle^{\perp}$.

Now, suppose that $D$ is a set of $k-r$ points. By induction, for $x \in \cA \setminus D$, we can choose an $a_x\in \bF_q$ so that $\phi-a_x\phi'$ is zero on $\langle D \cup \{ x \}\rangle^{\perp}$. 
For two distinct points $x,y\in \cA\setminus D$ the form $a_x\phi' - a_y\phi'$ therefore vanishes on the two hyperplanes $\langle D \cup \{ x \}\rangle^{\perp}$ and $\langle D \cup \{ y \}\rangle^{\perp}$ of $\langle D\rangle ^\perp$. Since $\phi'$ does not vanish on their intersection (there exists a hyperplane through $D$, $x$ and $y$ which is spanned by $\cA$) it follows that $a_x=a_y$. So there exists an $a\in \bF_q$ for which the form $\phi-a\phi'$ vanishes on every hyperplane  of $\langle D\rangle^\perp$ of the form $\langle D \cup \{ x \}\rangle^{\perp}$, with $x\in \cA\setminus D$. By hypothesis there are at least $mt+1$ such hyperplanes. It follows that $\phi-a\phi'$ must be zero on $\langle D\rangle ^\perp$. This proves the inductive step.
\end{proof}

The following theorem is a corollary of the uniqueness part of Theorem~\ref{BBTthm}. In the planar case this is due to Segre \cite{Segre1967} and the general case is due to Blokhuis, Bruen and Thas \cite{BBT1988} and \cite{BBT1990}.

\begin{theorem}  \label{uniqueextension}
Let $m\in \{1,2 \}$ such that $m-1=q$ modulo $2$. If $\cA$ is an arc of $\mathrm{PG}(k-1,q)$ of size at least $mq/(m-1)+k-1$ then $\cA$ has a unique completion to a complete arc.
\end{theorem}

\begin{proof} 
Let $\phi(Z)$ be the form of degree $mt$ obtained from $\cA$ by Theorem~\ref{BBTthm}. Suppose that $\cA$ can be extended to a larger arc by appending the point $x$. In the dual space the hyperplane $x^{\perp}$ contains the point dual to the hyperplane $\langle S\cup \{x\}\rangle^{\perp}$ for every $(k-2)$-subset $S$ of $\cA$. By induction on $r$, we will prove that $\phi$ is zero on $\langle D\cup x\rangle ^{\perp}$ for every subset $D$ of $\cA$ of size $k-2-r$. Then, for $r=k-2$, this will imply that $\phi$ is zero on $x^{\perp}$. Thus, $\phi$ contains linear factors of multiplicity $m$ for each point $x$ which extends $\cA$ to a larger arc. Thus, the extension of $\cA$ to a complete arc is unique and can be found by finding the linear factors of $\phi(Z)$.

To prove the induction claim for $r=1$, let $D$ be a $(k-3)$-subset of $\cA$. Then, for every point $y \in \cA \setminus D$, the point $\langle D\cup \{ x,y \}\rangle^{\perp}$ is a point on the line $\langle D\cup \{ x \}\rangle^{\perp}$, which is a zero of $\phi$ of multiplicity $m$. Since $\phi$ has degree $mt$, this implies that $\phi$ is zero on the line $\langle D\cup \{ x \}\rangle^{\perp}$.

To prove the inductive step, consider a $(k-2-r)$-subset of $D$ of $\cA$ 
and observe that for every point $y \in \cA \setminus D$, the form $\phi$ is zero on $\langle D\cup \{ x,y \}\rangle^{\perp}$. Thus, $\phi$ is zero on more than $mt$ hyperplanes of $\langle D\cup \{ x \}\rangle^{\perp}$, which implies $\phi$ is zero on  $\langle D\cup \{ x \}\rangle^{\perp}$.
\end{proof}

The following theorem is a corollary to Theorem~\ref{BBTthm}.

\begin{theorem} \label{dualhypersurface}
Let $m\in \{1,2 \}$ such that $m-1=q$ modulo $2$. Let $\cA$ be an arc of $\mathrm{PG}(k-1,q)$ of size $q+k-1-t$, where $|\cA| \geq mt+k-1$. Let $\mathcal T$ be the set of points in the dual space which are dual to the hyperplanes of $\mathrm{PG}(k-1,q)$  which contain exactly $k-2$ points of $\cA$. Then $\mathcal T$ is contained in a unique hypersuface of degree $mt$.
\end{theorem}

\begin{proof}
The hypersurface is defined by $\phi(Z)=0$, where $\phi(Z)$ is as in Theorem~\ref{BBTthm}.
\end{proof}

\section{A tensor associated to an arc}

The following theorem is in the spirit of Theorem~\ref{BBTthm}. It is stronger in that one does not require $m=2$ in the case of $q$ odd, but weaker in the sense that the multi-homogeneous form $F$ does not necessarily come from a form $\phi(Z)$ on $\PG(k-1,q)$ under substitution. This theorem is from \cite{BL2018} and it is proved using the scaled coordinate-free lemma of tangents (Lemma~\ref{lem:g_functionpre}).

\begin{theorem} \label{BLform}
Let $\cA$ be an arc of $\mathrm{PG}(k-1,q)$ of size $q+k-1-t$ and let $\Psi[X]$ denote the subspace of homogenous polyomials of degree $t$ which are zero on $\cA$. There is a form $F(X_1,\ldots,X_{k-1})$, which is homogenous of degree $t$ in each of the indeterminates $X_i=(X_{i1},\ldots,X_{ik})$ with the property that for each $S=\{ y_1,\ldots,y_{k-2} \}\subset \cA$
$$
F(X,y_1,\ldots,y_{k-2})=f_{S}(X) \pmod{\Psi[X]}.
$$ 
Moreover, modulo $(\Psi[X_1],\ldots,\Psi[X_{k-1}])$, the form $F$ is alternating for $t$ even and symmetric for $t$ odd.
\end{theorem}

One can show that for large subsets of $\cA$ such a form exists such that the statement holds without the modulo $\Psi[X]$. We conjecture that the statement indeed holds without the modulo $\Psi[X]$. 

The multi-homogeneous form $F$ from Theorem \ref{BLform} is called a {\em tensor form of $\cA$}. Here is an example.

\begin{example} \label{12arc}
\rm{The set of 12 points in $\mathrm{PG}(2,13)$,
$$
\cA=\{ ( 3,4,1),(-3,4,1),(3,-4,1),(-3,-4,1),(4,3,1),(4,-3,1),(-4,3,1),(-4,-3,1),
$$
$$
(1,1,1),(1,-1,1),(-1,1,1),(-1,-1,1) \}
$$
is an arc with $t=3$ and it is not contained in a curve of degree $3$. Consequently, Theorem~\ref{BLform} implies that there is bi-homogeneous form $F(X,Y)$ of degree three with the property that $F(X,a)=f_a(X)$ for all $a \in S$. It is given by
$$
F(x,y)=5(x_2^2x_3y_1^2y_3+y_2^2y_3x_1^2x_3+x_2x_3^2y_1^2y_2+x_1^2x_2y_2y_3^2+x_1x_3^2y_1y_2^2+x_1x_2^2y_1y_3^2)$$
$$
+6 x_1x_2x_3y_1y_2y_3+x_1^3y_1^3+x_2^3y_2^3+x_3^3y_3^3.
$$
In this example the tensor form $F$ of $\cA$ is unique up to a scalar factor.}
\end{example}

We will see that Theorem \ref{BLform} has strong implications for planar arcs and leads to a proof of the MDS conjecture for $k\leq \sqrt{q} - \sqrt{q}/p+2$ in the case that $q$ is an odd square.

The tensor form for planar arcs is a bi-homogeneous form of bi-degree $(t,t)$, which we call a {\em $(t,t)$-form}. To be able to state the next theorem we need the following terminology.
Let $V_r[X]$ denote the vector space of homogeneous polynomials of degree $r$ in ${\mathbb{F}}_q[X_1,X_2,X_3]$.
For a set $D$ of points of $\mathrm{PG}(2,q)$, we define an {\em $r$-socle} of $D$ as follows.
Let $\mathrm{M}_r(D)$ denote the 
matrix  whose rows are indexed by the elements of $V_r[X]$ and whose columns are indexed by the points of $D$, and where the $(f(X),x)$-entry is $f(x)$. Then an {\em $r$-socle} of $D$ is a subset of 
$D$ whose elements index the columns which form a basis for the column space of $\mathrm{M}_r(D)$.

The following theorem from \cite{BL2018} applies to planar arcs and is more detailed than Theorem \ref{BLform} for the tensor form of an arc in higher dimensions.

\begin{theorem} \label{thm:tensor_planar}
Let $\cA$ be an arc of size $q+2-t$ of $\mathrm{PG}(2,q)$ and let $\Psi[X]$ denote the subspace of homogeneous polyomials of degree $t$ which are zero on $\cA$. For any subset $S$ of $\cA$ containing a $t$-socle $S_0$ of $\cA$, where 
$$
|S \setminus S_0|= \min\{ \tfrac{1}{2}(t+2)(t+1),|\cA \setminus S_0| \},
$$
there is a $(t,t)$-form $F(X,Y)\in {\mathbb{F}}_q[X,Y]$ such that
$$
F(X,y)=f_y(X) \pmod {\Psi[X]}
$$
for all $y \in \cA$ and 
$$
F(X,y)=f_y(X),
$$
for all $ y\in S$. Moreover, modulo $(\Psi[X],\Psi[Y])$, the form $F$ is alternating for $t$ even and symmetric for $t$ odd.
\end{theorem}

As we will see in the section on planar arcs, Theorem \ref{thm:tensor_planar} leads to the existence of certain algebraic curves containing a given planar arc.


\section{Planar arcs}

It follows from the trivial upper bound Theorem \ref{cor:trivial_upper_bound} that a planar arc can have size at most $q+2$, and by the existence of a hyperoval, this is best possible for $q$ even. However, when $q$ is odd this can easily be improved.

\begin{theorem}
If $q$ is odd then a planar arc has size at most $q+1$.
\end{theorem}
\begin{proof}
As we have seen in the proof of Theorem \ref{cor:trivial_upper_bound}, considering a point $x\in S$ and the $q+1$ lines through $x$, the bound $|\cA|\leq q+2$ follows immediately. Now suppose $|\cA|=q+2$. 
Then each line through a point in $\cA$ must contain another point in $\cA$, and therefore each line intersects $\cA$ in 0 or 2 points.
Counting points on the lines through a point $y\notin S$ implies that $|\cA|$, and hence $q$, is even.
\end{proof}

The following theorem is Segre's celebrated characterisation of conics from \cite{Segre1955a}. We give a proof using the tensor associated to a planar arc.

\begin{theorem}\label{thm:segre}
If $q$ is odd and $\cA$ is an arc of size $q+1$ in $\PG(2,q)$ then $\cA$ is a conic.
\end{theorem}
\begin{proof}
Applying Theorem \ref{thm:tensor_planar}, where in this case $t=1$ and $\Psi[X]=\{0\}$, we obtain a symmetric bilinear form $F(X,Y)$ with $F(a,a)=0$ for each $a\in \cA$.
\end{proof}

\subsection{Hyperovals} \label{hypsection}
The situation is very different for $q$ even. There are many examples of arcs of size $q+2$ known for $q$ even which are projectively inequivalent. These can be described by a so-called {\em $o$-polynomial} which is a polynomial $f$ with the property that
$$
\{(1,t,f(t)) \ | \ t \in {\mathbb F}_q \} \cup \{ (0,1,0),(0,0,1) \}
$$
is an arc of size $q+2$. A planar arc of size $q+2$ is called a {\em hyperoval}.

The list of hyperovals is complete for $q=8$, $q=16$ and $q=64$. There is a sporadic example when $q=32$ due to O'Keefe and Penttila \cite{OP92}.

Table 1 lists all the known families of hyperovals. The references
for these infinite classes of hyperovals are: (1) regular (Bose \cite{RCB:47}),
(2) translation (Segre \cite{Segre1957}), (3) Segre \cite{Segre1962},
(4) Glynn I and Glynn II \cite{DGG:83}, (4) Payne \cite{SEP:85},
(5) Cherowitzo \cite{WEC:95B}, (6) Subiaco (Cherowitzo, Penttila, Pinneri and Royle \cite{WEC:95A}, Payne \cite{SEP:95A}, Payne, Penttila and Pinneri \cite{SEP:95z}),
(7) Adelaide (Cherowitzo, O'Keefe and Penttila \cite{WEC:00}).

{\small
\[
\begin{array}{|c|c|c|c|}
\hline
\mbox{Name} & f(X) & q=2^h & \mbox{Conditions}  \\
\hline \hline
\mbox{Regular} & X^2 &  & \\  \hline
\mbox{Translation} & X^{2^i} & & (h,i)=1  \\ \hline
\mbox{Segre} & X^6 & h \mbox{ odd} &  \\  \hline
\mbox{Glynn I} & X^{3\sigma+4} & h \mbox{ odd} & \sigma=2^{(h+1)/2}  \\  \hline
\mbox{Glynn II} & X^{\sigma + \lambda} & h \mbox{ odd} & 
\sigma=2^{(h+1)/2}; \\
& & & \mbox{$\lambda=2^m$ if  $h=4m-1;$}  \\
& & & \mbox{$\lambda=2^{3m+1}$ if  $h=4m+1$} \\ \hline
\mbox{Payne}&  P(X) & h \mbox{ odd} &   \\   \hline
\mbox{Cherowitzo} & C(X) & h \mbox{ odd}
&  \sigma=2^{(h+1)/2}  \\ \hline
\mbox{Subiaco I} & S_1(X) & h=4r+2 &  \omega^2+\omega+1 =0 \\
  \hline
\mbox{Subiaco II} & S_2(X) & h=4r+2 &  \delta =\zeta^{q-1} + \zeta^{1-q},
 \\  
              &        &         & \mbox{$\zeta$ primitive in $\bF_{q^2}$}\\
\hline
\mbox{Subiaco III} & S_3(X) & h\neq 4r+2 &  T_2(1/\delta) =1 \\ 
  \hline
\mbox{Adelaide} & S(X) & h \mbox{ even}, & \beta \in \bF_{q^2} \setminus \{1\},\ \beta^{q+1}=1,  
\\ & &h \geq 4 & m \equiv \pm (q-1)/3 \pmod{q+1}  \\\hline
 \multicolumn{4}{c}{    } \\
\multicolumn{4}{c}{\mbox{Table 1: Hyperovals in } \mathrm{PG}(2,q), q \mbox{ even}}
\end{array}
\]}

In Table 1, $T_2$ denotes the trace function from ${\mathbb F}_q$ to ${\mathbb F}_2$ and 
$$
P(X) = X^{1/6}+X^{3/6}+X^{5/6},
$$
$$
C(X) = X^{\sigma}+X^{\sigma+2}+X^{3\sigma+4}
$$
$$
S_1(X) = \frac{\omega^2(X^4+X)}{X^4+\omega^2X^2+1} +X^{1/2},
$$
$$
S_2(X) =  \frac{\delta^2X^4+\delta^5X^3+\delta^2X^2+\delta^3X}{X^4+
\delta^2X^2+1}+\left(\frac{X}{\delta}\right)^{1/2},
$$
$$
S_3(X) =  \frac{(\delta^4+\delta^2)X^3+\delta^3X^3+\delta^2X}{X^4+\delta^2
X^2+1} + \left(\frac{X}{\delta}\right)^{1/2},
$$
and
$$
S(X) =\frac{T(\beta^m)(X+1)}{T(\beta)}+\frac{T((\beta X +\beta^q)^m)}{T(\beta)(X+T(\beta)X^{1/2}+1)^{m-1}}+X^{1/2},
$$
where $T(X)=X+X^q$.

The complete classification of hyperovals has been determined for $q \leq 64$ and is given in the next theorem.

\begin{theorem}

{\rm (i)} {\em (Segre \cite{Segre1957})} In $\mathrm{PG}(2,q)$, $q=2,4,8$, the only
hyperovals are the regular hyperovals.

{\rm (ii)} {\em (Hall \cite{MH:75}, O'Keefe and Penttila \cite{CMO:91})} In $\mathrm{PG(2,16)}$, there are exactly two
distinct hyperovals. They are the regular hyperoval and the Subiaco
hyperoval. This latter hyperoval is also called the {\em Lunelli-Sce}
hyperoval {\rm \cite{LL:64}}.

{\rm (iii)} {\em (Penttila and Royle \cite{TP:94})} In $\mathrm{PG}(2,32)$, there
are exactly six distinct hyperovals. They are the regular hyperoval, the
translation hyperoval, the Segre hyperoval, the Payne
hyperoval, the Cherowitzo hyperoval and the O'Keefe--Penttila hyperoval.

{\rm (iv)} {\em (Vandendriessche \cite{ Vandendriessche2019})} In $\mathrm{PG}(2,64)$, there
are exactly four distinct hyperovals. They are the regular hyperoval, the Subiaco I, the Subiaco II and the Adelaide hyperoval.
\end{theorem}

The following theorem is due to Caullery and Schmidt and is from \cite{CS2015}. 

\begin{theorem}
If $f$ is an $o$-polynomial of ${\mathbb F}_q$ of degree less than $\frac{1}{2}q^{1/4}$ then $f(X)$ is equivalent to either $X^6$ or $X^{2^k}$ for a positive integer $k$.
\end{theorem}

This also classifies {\em exceptional $o$-polynomials} extending results from \cite{HeMc2012} and \cite{Zieve2015} on monomial $o$-polynomials.

\subsection{Second largest planar arcs, $q$ even}

The following theorem implies that the second largest complete arc is somewhat smaller than the size of a hyperoval.
Its proof illustrates the power of the algebraic hypersurface (in this case a curve) of an arc from the planar version of Theorem \ref{BBTthm}.
\begin{theorem} \label{evenextention}
If $q$ is even then a planar arc of size at least $q-\sqrt{q}+\frac{3}{2}$ in $\PG(2,q)$ is extendable to a hyperoval.
\end{theorem}

\begin{proof}
Let $t=q+2-|\cA|\leq \sqrt{q}+\frac{1}{2}$, and observe that $\cA$ satisfies the hypothesis of Theorem  \ref{uniqueextension} (since $q$ is even, $m=1$).
It follows that $\cA$ is uniquely extendable to a complete arc.

So we may assume that $\cA$ is a complete arc of size at least $q-\sqrt{q}+\frac{3}{2}$. Suppose $\cA$ has size less than $q+2$. Then $t\geq 1$. Let $\phi(Z)$ be the form of degree $t$ obtained from $\cA$ by Theorem \ref{BBTthm}.
If $\phi(Z)$ has a linear factor then by the arguments used in the proof of Theorem \ref{uniqueextension}, 
$\cA$ can be extended to a larger arc, a contradiction.
Hence $\phi(Z)$ does not have linear factors. Let $T$ denote the set of points on the curve defined by $\phi(Z)$. Then $T\neq \emptyset$, 
and each line meets $T$ in at most $t$ points. Choose $x\in T$.
Counting points of $T $ on the lines through $x$ we obtain $|T |\leq (t-1)(q+1)+1$.
Since the tangents to $\cA$ form a set of $|\cA|t=(q+2-t)t$ points in the dual plane, we have 
$$(q+2-t)t \leq (t-1)(q+1)+1,$$
which contradicts $t\leq \sqrt{q}+\frac{1}{2}$.
\end{proof}

In \cite{Segre1967} Segre used the algebraic curve in the dual plane of a planar arc (which he called the {\em algebraic envelope}) in combination with the Hasse-Weil bound to 
obtain the upper bound $q-\sqrt{q}+1$ for a planar arc in $\PG(2,q)$, $q$ even. This bound is sharp as we will see later.

\begin{theorem} \label{evenextention}
If $q$ is even then a complete planar arc in $\PG(2,q)$ which is not contained in a hyperoval has size at most $q-\sqrt{q}+1$.
\end{theorem}

The following theorem is from Voloch \cite{Voloch1991}. Its proof combines Theorem \ref{BBTthm} with the St\"ohr-Voloch theorem.

\begin{theorem} \label{volocheven}
If $q$ is an even non-square then a planar arc of size larger than $q-\sqrt{2q} +2$ is extendable to a hyperoval.
\end{theorem}

\subsection{Second largest planar arcs, $q$ odd}
The bounds on the size of the second largest arc in planes of odd order, obtained by the methods described above, are worse than for $q$ even. The reason for this is the fact that the hypersurface from Theorem \ref{BBTthm} has degree $2t$ for $q$ odd.
As with the proof of Theorem~\ref{volocheven}, the theorems by Voloch use Theorem~\ref{dualhypersurface} and the St\"ohr-Voloch theorem from \cite{SV1986}. The following theorems are from \cite{Voloch1990b} and  \cite{Voloch1991} respectively.

\begin{theorem} \label{primeplane}
If $q$ is prime then a planar arc of size larger than $\frac{44}{45}q+\frac{8}{9}$ is contained in a conic.
\end{theorem}

\begin{theorem} \label{oddnonsq}
If $q$ is an odd non-square then a planar arc of size larger than $q-\frac{1}{4}\sqrt{pq} +\frac{29}{16}p-1$ is contained in a conic.
\end{theorem}

In contrast to the previous bounds the following theorem does not rely on Hasse-Weil or St\"ohr-Voloch.
It is proved using the tensor of a planar arc Theorem~\ref{thm:tensor_planar} from \cite{BL2018}. It illustrates that in this case the tensor form (of degree $t$ in each component)  is stronger than the algebraic hypersurface (which is of degree $2t$).

\begin{theorem} \label{twocurves}
Let $\cA$ be a planar arc of size $q+2-t$ not contained in a conic. If $q$ is odd then $\cA$ is contained in the intersection of two curves, sharing no common component, each of degree at most $t+p^{\lfloor \log_p t \rfloor}$.
\end{theorem}

Theorem~\ref{twocurves} has the following corollary.

\begin{theorem} \label{oddsq}
If $q$ is an odd square then a planar arc of size at least $q-\sqrt{q}+\sqrt{q}/p+3$ is contained in a conic.
\end{theorem}

In the case that $q$ is a non-square and non-prime, the bound obtained from Theorem~\ref{twocurves} does not improve upon the bound of Voloch mentioned above. However, in the case that $q$ is prime, it does improve on Voloch's bound for primes less than 1783, see \cite[Theorem 5]{BL2018}.

\subsection{Planar arcs of size $q-\sqrt{q}+1$}

In the case that $q$ is a square, Kestenband \cite{Kestenband1981} constructed large complete planar arcs as the intersection of Hermitian curves. For any $3 \times 3$ matrix $\mathrm{M}$, let 
$$
V(\mathrm{M})=\{ x \in \PG(2,q) \ | \ x^t \mathrm{M} x^{\sqrt{q}} =0 \}.
$$

\begin{theorem} \label{kesetenband} 
Let $q> 4$ be a square. Let $\mathrm{I}$ be the $3 \times 3$ identity matrix. Let $\mathrm{H}=(a_{ij})$ be a $3 \times 3$ matrix with the property that $\mathrm{H}^{\sqrt{q}}=\mathrm{H}^t$, where $\mathrm{H}^{\sqrt{q}}=(a_{ij}^{\sqrt{q}})$ and $\mathrm{H}^t=(a_{ji})$.

If the characteristic polynomial of $\mathrm{H}$ is irreducible over ${\mathbb F}_q$ then the intersection of the Hermitian curves $\cA=V(\mathrm{I}) \cap V(\mathrm{H})$ is a planar arc of size $q-\sqrt{q}+1$ which is not contained in a conic.
\end{theorem}

\begin{proof}
Consider the Hermitian curves $V(\mathrm{H} +\mu \mathrm{I})$, where $\mu \in{\mathbb F}_{\sqrt{q}}$ and $V(\mathrm{I})$. 
If $x$ is a point on two of these curves then $x \in \cA$.
If $x \not\in \cA$ and $x \not \in V(\mathrm{I})$ then $x$ is a point of $V(\mathrm{H} +(a/b) \mathrm{I})$, where $x^t \mathrm{H} x^{\sqrt{q}} =a$ and $x^t \mathrm{I} x^{\sqrt{q}} =-b$. 
Hence, each point is either in $\cA$ or on exactly one of the $\sqrt{q}+1$ Hermitian curves.
Therefore, 
$$
(\sqrt{q}+1)(q\sqrt{q}+1)=|\cA|(\sqrt{q}+1)+q^2+q+1-|\cA|,
$$
which gives $|\cA|=q-\sqrt{q}+1$.

Suppose that $\ell$ is a line incident with $r \geq 2$ points of $\cA$. Then $\ell$ intersects each Hermitian curve ($V(\mathrm{H} +\mu \mathrm{I})$ or $V(\mathrm{I})$) in $\sqrt{q}+1$ points, $r$ of which are in $\cA$ and $\sqrt{q}+1-r$ of which are not in $\cA$.

Counting points of $\ell$ not in $\cA$ we have $(\sqrt{q}+1-r)(\sqrt{q}+1)=q+1-r$, since each point not in $\cA$ is on exactly one of the $\sqrt{q}+1$ Hermitian curves. This gives $r=2$ and so $\cA$ is an arc.

It follows from Bezout's theorem that $\cA$ has at most $2 \sqrt{q}+2$ points in common with a conic, so cannot be contained in a conic for $q-\sqrt{q}+1 \geq 2\sqrt{q}+2$. The case $q=9$ can be checked directly.
\end{proof}

\subsection{Classification of planar arcs of size $\geq q-2$}

Planar arcs of size $q$ were shown to be incomplete by Segre for  $q$ odd and by Tallini for $q$ even. In contrast there do exist complete planar arcs of size $q-1$. 
The classification of planar arcs of size $q-1$ and $q-2$ was only recently completed in \cite{BL2018} and is given in the following corollaries.

Corollary \ref{qminus1} proves a conjecture from Hirschfeld, Korchm\'aros and Torres \cite[Remark 13.63]{HiKoTo2008}. It was proven in \cite{BL2018} as a corollary of Theorem \ref{twocurves}.
\begin{corollary} \label{qminus1}
The only complete planar arcs of size $q-1$ occur for $q=7$ (there are $2$ projectively distinct arcs of size $6$), for $q=9$ (there is a unique arc of size $8$), $q=11$ (there is a unique arc of size $10$) and $q=13$ (there is a unique arc of size $12$). 
\end{corollary}

\begin{corollary} \label{qminus2}
The only complete planar arcs of size $q-2$ occur for $q=8$ (there are $3$ projectively distinct arcs of size $6$), for $q=9$ (there is a unique arc of size $7$) and $q=11$ (there are $3$ projectively distinct arcs of size $9$). 
\end{corollary}

\section{A proof of the MDS conjecture for $k\leq \sqrt{q} - \sqrt{q}/p+2$}
Suppose that $q$ is an odd square. In the previous section we saw how the tensor for planar arcs has far reaching consequences for the size of planar arcs which are not contained in conics. Using the projection theorem this leads to a proof of the MDS conjecture for $k\leq \sqrt{q} - \sqrt{q}/p+2$.
The following two corollaries are from \cite{BL2018}.

\begin{corollary} \label{cor_main_conjecture}
If $k\leq \sqrt{q}-\sqrt{q}/p+1$ and $q=p^{2h}$, $p$ odd, then an arc of $\mathrm{PG}(k-1,q)$ of size $q+1$ is a normal rational curve. 
\end{corollary}

\begin{proof}
Let $\cA$ be an arc of size $q+1$ in $\mathrm{PG}(k-1,q)$. The projection of $\cA$ from a subset of size $k-3$ is a planar arc of size at least $q+1-(\sqrt{q}-\sqrt{q}/p-2)$. By Theorem~\ref{oddsq}, this projection is contained in a conic. By Corollary \ref{cor:main_projection}, this implies that $\cA$ is a normal rational curve. 
\end{proof}

\begin{corollary} \label{cor_main_conjecture_2}
If $k\leq \sqrt{q}-\sqrt{q}/p+2$ and $q=p^{2h}$, $p$ odd, then an arc of $\mathrm{PG}(k-1,q)$ has size at most $q+1$. 
\end{corollary}
\begin{proof}
Suppose $\cA$ is an arc of size $q+2$ in $\PG(k-1,q)$. By Corollary \ref{cor_main_conjecture}, the projection of $\cA$ from any of its points is a normal rational curve.
By Theorem \ref{thm:main_projection}, the arc $\cA$ is contained in a normal rational curve, a contradiction.
\end{proof}

\section{Arcs in $\PG(3,q)$}\label{sec:k=4}

The following theorem was proved by Segre \cite{Segre1955b} for $q$ odd and by Casse \cite{Casse1969} for $q$ even. The theorem proves the MDS conjecture for $k=4$.

\begin{theorem}
An arc of $\PG(3,q)$, $q\geq 4$, has size at most $q+1$.
\end{theorem}
\begin{proof}
If $q$ is odd and $\cA$ is an arc in $\PG(3,q)$ of size $q+2$, then by projection from a point of $\cA$ we obtain an arc of size $q+1$ in a plane, which by Segre's theorem is a conic. By projecting $\cA$ from two distinct points and applying Theorem \ref{thm:main_projection} it follows that $\cA$ is contained in a unique normal rational curve of $\PG(3,q)$. This contradicts the assumption $|\cA|=q+2$.

Let $\cA$ be an arc of size $q+2$ with $q$ even. By Theorem~\ref{dualhypersurface}, the set of points dual to the planes containing exactly two points of $\cA$
is contained in a hypersurface of degree one (a plane). Hence, these planes are all concurrent in a point $y$, which extends $\cA$ to an arc of size $q+3$. 

Let $\cA'$ be an arc of size $q+3$. 
The number of points which are incident with bisecants to $\cA'$ is
$$
q+3+{q+3 \choose 2}(q-1).
$$
This is less than $(q+1)(q^2+1)$, the total number of points of $\PG(3,q)$, which implies there is a point $u$ not incident with any bisecant.
Since $|\cA'|=q+3$, any plane incident with two points of $\cA'$ is incident with three points of $\cA'$. Therefore, the number of planes incident with $u$ and three points of $\cA'$ is
$$
\frac{1}{3} {q+3 \choose 2},
$$ 
which implies $3$ divides $q+2$.

Let $x$ and $y$ be points of $\cA'$ spanning the line $\ell$ and let $v$ be a point of $\ell \setminus \{x,y\}$. For each pair of points $z,w \in \cA' \setminus \{x,y\}$, the plane spanned by $v$, $z$ and $w$ contains three points of $\cA'\setminus \{x,y\}$.
Thus, the number of $3$-secant planes which are incident with $v$ but do not contain $x$ or $y$ is
$$
\frac{1}{3} {q+1 \choose 2}.
$$
This implies $3$ divides $q+1$, contradicting the previously obtained divisibility condition.
\end{proof}

The following theorem classifies arcs of size $q+1$ in $\PG(3,q)$. The case $q$ odd was already obtained by Segre \cite{Segre1955b}.
The case $q$ even is due to Casse and Glynn \cite{CG1982}.

\begin{theorem}\label{thm:q+1_arc_in_3space}
Let $\cA$ be an arc in $\PG(3,q)$ of size $q+1$. If $q$ is odd and $q\geq 5$ then $\cA$ is a twisted cubic. If $q$ is even and $q\geq 8$ then $\cA$ is projectively equivalent to Example \ref{ex:3space}.
\end{theorem}
\begin{proof}
Suppose $q$ is odd. Projecting $\cA$ from two distinct points gives to planar arcs of size $q$, each of which, by Theorems~\ref{oddsq}, \ref{primeplane}, \ref{oddnonsq}, is contained in a conic. Applying Theorem \ref{thm:main_projection} concludes the proof.

Suppose $q=2^h$.
Projecting $\cA$ from a point $x\in \cA$ gives a planar arc $x(\cA)$ of size $q$. By Theorem~\ref{evenextention}, $x(\cA)$ is extendable to a hyperoval. Let $a,b$ be the unique pair of points which extend $x(\cA)$ to this hyperoval. The lines $xa$ and $xb$ are called the {\em tangent lines of $\cA$ at the point $x$}. For any point $y\neq x$ of $\cA$, it follows from Theorem~\ref{dualhypersurface}, that the planes $xya$ and $xyb$ are contained in a hyperbolic quadric in the dual space. Therefore the lines $xa$ and $xb$ are contained in a hyperbolic quadric in the dual space. Hence, the $2(q+1)$ tangent lines of $\cA$ are the lines on a hyperbolic quadric $Q^+$ in $\PG(3,q)$.
We may assume that $Q^+$ is the zero locus of the quadratic form $X_1X_4-X_2X_3$ and that $x=(1,0,0,0)$ and $y=(0,0,0,1)$ are in $\cA$. The ovals $x(\cA)$ and $y(\cA)$ correspond to two {\em o-polynomials} $f$ and $g$ with $f(0)=g(0)=0$ and $f(1)=g(1)=1$. The arc $\cA$ can therefore be written as
$$\cA=\{(rg(r),g(r),r,1)~:~r\in \bF_q\} \cup \{x\}=\{(1,s,f(s),sf(s))~:~s\in \bF_q\} \cup \{y\}.$$
Equating coefficients in the $o$-polynomilals $f$ and $g$, leads to the condition that $f(1/s)=g(s)$, for each $s\neq 0$, and consequently to the conclusion that the only possibility for $f$ and $g$ is $f(s)=g(s)=s^n$ for some integer $n$. So 
$$\cA=\{(1,s,s^n,s^{n+1})~:~s\in \bF_q\} \cup \{(0,0,0,1)\}.
$$
It follows that $\cA$ is invariant under the projectivity $\varphi$ interchanging the first with the last and the second with the third coordinate. In particular $\varphi$ interchanges the points $x$ and $y$. Since $x$ and $y$ were arbitrary points of $\cA$ it follows that $\cA$ is invariant under a group $G\leq \PGL(4,q)$, which acts transitively on the points of $\cA$. Since the tangent plane at $x$ (which has equation $X_4=0$) intersects $\cA$ in $x$ with multiplicity $n+1$, each tangent plane must intersect $\cA$ in a point with multiplicity $n+1$. Imposing this condition on the tangent plane at a general point $(1,s,s^n,s^{n+1})$ implies that $n=2^e$ for some integer $e$. The fact that $\cA$ is an arc implies $(e,h)=1$ (see last line of the proof of Example \ref{ex:3space}).
\end{proof}

\section{Arcs in $\PG(4,q)$}\label{sec:k=5}

The following theorem proves the MDS conjecture for $k=5$. It was proved by Segre \cite{Segre1967} for $q$ odd and Casse \cite{Casse1969} for $q$ even.
\begin{theorem}
An arc $\cA$ in $\PG(4,q)$, $q\geq 5$, has size at most $q+1$.
\end{theorem}
\begin{proof}
Let $\cA$ be an arc in $\PG(4,q)$ of size $q+2$. 

Suppose $q$ is odd. Projecting $\cA$ from any two distinct points of $\cA$ gives a planar arc of size $q$, which is contained in a conic. By Corollary  \ref{cor:main_projection}, the arc $\cA$ is contained in a normal rational curve, a contradiction.

Suppose $q$ is even.
Projecting from two distinct points $x,y\in \cA$ gives two arcs $x(\cA)$ and $y(\cA)$ of size $q+1$ in a hyperplane $\pi$ of $\PG(4,q)$.
Arguing as in the proof of Theorem \ref{thm:q+1_arc_in_3space}, it follows that for each of these projected arcs the
$2(q+1)$ tangent lines are the lines on a hyperbolic quadric $Q_x^+$ and $Q_y^+$ in $\pi$. In particular, $x(\cA)$ is contained in $Q_x^+$ and $y(\cA)$
is contained in $Q_y^+$. 

Let $z$ be the point of intersection of the line through $x$ and $y$ with $\pi$ and let $\ell$ be one of the tangent lines of $x(\cA)$ at $z$. 
By the definition of $\ell$, each plane through $\ell$ contains at most one point of $x(\cA)$ different from $z$.
For each $p\in \cA\setminus\{x,y\}$ the plane $\langle x,y,p\rangle$ meets $\pi$ in a line containing $z$, $x(p)$ and $y(p)$. Therefore, the plane spanned by $\ell$ and $x(p)$ also contains $y(p)$. Hence, each plane through $\ell$ also contains at most one point of $y(\cA)$ different from $z$. It follows that the two tangents, say $\ell$ and $m$, of $x(\cA)$ at $z$ are equal to the two tangents of $y(\cA)$ at $z$.

The quadratic cones $xQ_x^+$ and $yQ_y^+$ therefore coincide on the hyperplane $\langle x,\ell,m\rangle$ with the union of the two planes $\langle x,\ell\rangle$ and $\langle x,m\rangle$. After a suitable choice of basis, we can assume that $xQ_x^+$ and $yQ_y^+$ are the zero locus of the forms $X_1X_2+X_5a(X)$ and $X_1X_2+X_5b(X)$, where $a(X)$ and $b(X)$ are linear forms in $X=(X_1,\ldots,X_5)$. The arc $\cA$ is therefore contained in the zero locus of $X_5(a(X)-b(X))$, which is the union of two hyperplanes.
This implies that $\cA$ has size at most $8$, a contradiction.
%
\end{proof}
The following improves on the previous bound $q>83$ (see e.g. \cite[Table 3.1]{HiSt2001}).
\begin{theorem}
An arc in $\PG(4,q)$, $q$ odd, $q>13$, of size $q+1$ is a normal rational curve.
\end{theorem}
\begin{proof}
Projecting from two points gives a planar arc of size $q-1$ which by Theorem \ref{qminus1} is contained in a conic unless $q\leq 13$.
\end{proof}

\begin{remark}
An arc $\cA$ of size 8 in $\PG(4,7)$ gives a 5-dimensional MDS code $C_\cA$ of length 8, whose dual is a 3-dimensional MDS code of length 8, which by Segre's theorem is a Reed-Solomon code. By Theorem \ref{thm:dual_of_RS}, the code $C_\cA$ is also a Reed-Solomon code, and therefore $\cA$ is a normal rational curve.
In \cite{Glynn1986} Glynn proved that there are exactly two arcs of size $10$ in $\PG(4,9)$.
\end{remark}

To prove the following theorem from \cite{CG1984}, Casse and Glynn prove that every arc of size $q$ in $\PG(3,q)$, $q$ even and $q \geq 16$, is incomplete and its completion is unique. From this one can then prove that an arc of size $q+1$ in $\PG(4,q)$, $q$ even and $q \geq 16$, is a normal rational curve.

\begin{theorem}\label{thm:NRC_PG(4,q)_q_even}
An arc in $\PG(4,q)$, $q=2^h$, $h\geq 3$ , of size $q+1$ is a normal rational curve.
\end{theorem}

%
%
%
%
%
%
%
%
%

\section{State of the art on the MDS conjecture}\label{sec:state_of_the_art}

Proofs of the MDS conjecture for $k$-dimensional codes with $k\in \{4,5\}$ can be found in Sections \ref{sec:k=4} and \ref{sec:k=5}.
The following table contains a survey of the conditions on the parameters for $k\geq 6$ for which the MDS conjecture has been proven.

\bigskip

{\small
\[
\begin{array}{|c|c|c|}
\hline
\mbox{Reference} & \mbox{condition on $q$} & \mbox{condition on $k$}    \\
\hline \hline

\mbox{Casse - Glynn 1984} & q=2^h & k=6   \\  \hline
\mbox{Kaneta - Maruta 1991} & q=2^h & k=7   \\  \hline

\mbox{Storme - Thas 1993} & q=2^h & k<\frac{1}{2}\sqrt{q}+\frac{15}{4}   \\  \hline

\mbox{Voloch 1991} & q =p^{2h+1} \mbox{ odd}&k<\frac{1}{4}\sqrt{pq} -\frac{29}{16}p+4   \\ \hline
\mbox{Ball - Lavrauw 2018} & q =p^{2h}  \mbox{ odd}& k<\sqrt{q}-\sqrt{q}/p+2  \\ \hline

\mbox{Hirschfeld - Korchm\'aros 1996} & q =p^h, p\geq 5  \mbox{ odd}& k\leq \frac{1}{2}\sqrt{q}     \\ \hline
\mbox{Hirschfeld - Korchm\'aros 1998} & q\geq 529 \mbox{ odd}, q\neq 3^6,5^5 & k\leq \frac{1}{2}\sqrt{q}+2    \\ \hline

\mbox{Ball 2012} &  q=p^h& k<p     \\ \hline
\mbox{Ball - De Beule 2012} & q=p^h, h>1 & k\leq 2p-2  \\  \hline

\multicolumn{3}{c}{    } \\
\multicolumn{3}{c}{\mbox{Parameters for which the MDS conjecture holds true in $\PG(k-1,q)$ with $k\geq 6$}}
\end{array}
\]}

The proof of the MDS conjecture for $k=6$, $q$ even, follows from Theorem \ref{thm:NRC_PG(4,q)_q_even} and Corollary \ref{cor:kaneta_maruta}.

In 1991 Kaneta and Maruta \cite{KaMa1991} proved the uniqueness of arcs of size $q+1$ in $\PG(5,q)$, $q$ even, $q\geq 16$, which by Corollary  \ref{cor:kaneta_maruta} proves
the MDS conjecture for $k=7$ and these values of $q$. 

In \cite{StTh1993} the Segre-Blokhuis-Bruen-Thas hypersurface is used to obtain an upper bound for the second largest arc in $\PG(4,q)$, $q$ even, improving upon the bound from \cite{BBT1988}. By Theorem \ref{thm:NRC_PG(4,q)_q_even}, any arc exceeding that bound is contained in a normal rational curve. Using Corollaries \ref{cor:main_projection} and \ref{cor:kaneta_maruta}, this proves the MDS conjecture for  $k<\frac{1}{2}\sqrt{q}+\frac{15}{4} $.

The following instances for the MDS conjecture for $q$ odd follow from the results on planar arcs and the projection theorem. Precisely, projecting an arc $\cA$ in $\PG(k-2,q)$ of size $q+1$ from $k-4$ of its points gives a planar arc of size $q-k+5$. So if $q+5-k>m$, i.e. $k<q+5-m$, and each planar arc of size $>m$ is contained in a conic then, by Corollary \ref{cor:main_projection}, $\cA$ is a normal rational curve. By Corollary \ref{cor:kaneta_maruta} the MDS conjecture holds true in $\PG(k-1,q)$.

The best asymptotic bounds for $q$ odd are from \cite{Voloch1991} and \cite{BL2018}. Theorem \ref{oddnonsq} implies the MDS conjecture for 
$k<\frac{1}{4}\sqrt{pq} -\frac{29}{16}p+4$, $q$ an odd non-square. If $q$ is an odd square then Theorem \ref{oddsq} implies the MDS conjecture for $k<\sqrt{q}-\sqrt{q}/p+2$.

For a proof of the MDS conjecture for $k\leq p$ and $k\leq 2p-2$, we refer to Sections \ref{sec:MDS_p} and \ref{sec:MDS_2p-2}.

The bounds on planar arcs from Hirschfeld and Korchm\'aros \cite{HiKo1996} and \cite{HiKo1998} imply the MDS conjecture for $k\leq \frac{1}{2}\sqrt{q}$, provided that the characteristic is at least 5, and for $k\leq \frac{1}{2}\sqrt{q}+2$, for $q\geq 529$ and $q\neq 3^6,5^5$. 

We have not dealt here very deeply with the techniques used to obtain the results of Voloch \cite{Voloch1991} and Hirschfeld and Korchm\'aros \cite{HiKo1996}, \cite{HiKo1998}. These results not only use bounds on the number of points of algebraic curves of small degree but also require more careful analysis of the curves which can arise from Segre's algebraic envelope associated to the tangents of a planar arc.

\end{document}

%% file: def_t.txt
\begin{center}
\begin{tikzpicture}[scale=0.4]  
    \def\rx{1}
    \def\ry{2}
    \draw [name path=A] [line width=1pt] (0,0) ellipse ({\rx} and {\ry});
    \node (A) at (0,0) {$S$};
    \path [name path=vertical] (0,-\ry) -- (0,\ry); 
    \path [name intersections = {of = A and vertical}];
    \draw [name path=hyp1] [rounded corners=10mm] (0,\ry) -- (8.5,\ry+1.5) -- (0.2,-\ry+0.005);
    \path [color=red,name path=testhyp2] (0,\ry) -- (9,\ry-1) -- (0,-\ry);
    \path [name intersections = {of = testhyp2 and hyp1}];
    \draw [name path=hyp2,rounded corners=10mm] (0.1,-\ry) -- (9,\ry-0.5) -- (intersection-2);
    \path [color=red,name path=testhyp3] (0,\ry) -- (9,-\ry-1) -- (0,-\ry);
    \path [name intersections = {of = testhyp3 and hyp2}];
    \draw [dotted,name path=hyp3,line width=1pt,rounded corners=10mm] (0,-\ry) -- (9,-\ry-1) -- (intersection-1);
    \path [color=red,name path=testhyp4] (0,\ry) -- (9,-\ry-3) -- (0,-\ry);
    \path [name intersections = {of = testhyp4 and hyp3}];
    \draw [dotted,name path=hyp4,line width=1pt,rounded corners=10mm] (0,-\ry) -- (9,-\ry-3) -- (intersection-1);
    \node (t) at (12,-3.2) {$t=q+1-|\cA\setminus S|$};
    \draw [name path=rest_of_S] [line width=1pt] [rotate around={20:(7,\ry)}] (7,\ry) ellipse ({\rx+1.4} and {\ry+0.2});
    \node (rest_of_S) at (8,\ry) {$\cA\setminus S$};
    \draw (6.5,\ry+0.7) node [draw,circle,fill=black,minimum size=3pt,inner sep=0pt]{};
    \draw (6.7,\ry-0.9) node [draw,circle,fill=black,minimum size=3pt,inner sep=0pt]{};
\end{tikzpicture}
\end{center}